\documentclass[11pt,reqno]{amsart}
\usepackage{amsmath, amsfonts, amsthm, amssymb, stmaryrd, color, cite}
\usepackage{parskip}
 \usepackage{hyperref}
 \usepackage{soul}
\usepackage{graphicx}
\usepackage{latexsym,hyperref}

\usepackage[dvipsnames]{xcolor}

\allowdisplaybreaks

\numberwithin{equation}{section}
\newtheorem{theorem}{Theorem}

\newtheorem{lemma}{Lemma}

\def\ve{\varepsilon}

\title[MDP for a Stochastic Schr\"odinger  Equation]
{{Moderate Deviations for a Stochastic Schr\"odinger   Equation with Linear Drift}}

\author[P. Fatheddin]{Parisa Fatheddin}
\address{Department of Mathematics, Ohio State University, USA.}
\email{fatheddin.1@osu.edu}

\author[H. Lisei]{Hannelore Lisei}
\address{Faculty of Mathematics and Computer Science, Babe\c{s}-Bolyai University, Cluj-Napoca, Romania}
\email{hannelore.lisei@ubbcluj.ro}


\begin{document}

\begin{abstract}
Moderate deviation principle is achieved by the weak convergence approach for a   stochastic Schr\"odinger type equation with linear drift term and noise driven by a $Q$-Wiener process. The central limit theorem is also shown for the equation to further analyze its asymptotic behavior.
\end{abstract}
\maketitle

\section{Introduction}
\indent Laser beam propagation through random media and many other phenomena in optics are modeled by Schr\"odinger type equations. Most of the study on these models in Physics has been concentrated on the deterministic equation. See for example, \cite{(19),(20),(36),(39)}. Here we study the stochastic counterpart. Denoting the space of complex numbers as $\mathbb{C}$, we consider the Sobolev   spaces, $H:= L^{2}(G;\mathbb{C})$ and $V:= H_{0}^{1}(G;\mathbb{C})$ for a bounded domain with smooth boundary, $G$ in $\mathbb{R}^{d}$ for any $1\leq d<\infty$, and we study the asymptotic behavior of the following equation when perturbed by small noise
\begin{align}\label{originalequation}
du^{\varepsilon}(t)&=  i \Delta u^{\varepsilon}(t)dt +\mathcal{U}(t)u^{\varepsilon}(t)dt + \sqrt{\varepsilon} g(t,u^{\varepsilon}(t))dW(t), t\in[0,T],\, \varepsilon\in(0,1),\\
u^{\varepsilon}(0) &=   \gamma   \in V. \nonumber
\end{align}
The potential $\mathcal{U} $ is a deterministic, bounded, complex-valued function (see Section \ref{sec2}) and $W$ is an $H$-valued $Q$-Wiener process. We note that the method in this paper  can also be adapted to hold for a noise driven by a cylindrical Wiener process. \\
\indent The existence and uniqueness of solution for a stochastic nonlinear Schr\"odinger equation (SNLSE) have been established by V. Barbu, M. R\"ockner and D. Zhang in the case of multiplicative noise in \cite{(4),(5)}.
For the global well-posedness of mild solutions see the results by A. Bouard and A. Debussche, \cite{(8)} for additive and \cite{(7), (8)} for multiplicative noise. Most of the authors investigating the well-posedness of SNLSE have concentrated on the mild solution by relying on the Strichartz estimates. A different approach is taken in \cite{Keller, (29)} where properties  of variational solutions are established. Notice that in the context of stochastic partial differential equations (SPDEs), each variational solution is also a mild solution (and also a weak solution) but not vice versa.  Equation \eqref{originalequation} is a special case of the stochastic Schr\"odinger equation considered in \cite{(29)}, where  the existence and uniqueness of  variational solutions  in space $L^{2}(\Omega;\mathcal{C}([0,T];H))\cap L^{2}(\Omega \times [0,T];V)$ are proved. Here we will use this result and prove the moderate deviation principle (MDP) and a central limit theorem in the space $\mathcal{C}([0,T];H)$ by letting $\varepsilon$ in \eqref{originalequation} tend to zero and investigate the convergence of the stochastic equation to its deterministic form. To the best of our knowledge, the results presented here are the first results on moderate deviations as well as the central limit theorem for any type of stochastic Schr\"odinger equation. We prove these two theorems in the space $\mathcal{C}([0,T];H)$ instead of the space $\mathcal{C}([0,T];H)\cap L^2(0,T;V)$ or $\mathcal{C}([0,T];V)$, since the term involving the   operator $i\Delta  $ (hence the term including the norm in $V$) turns into the term $\int_0^t\text{Im}\|u^{\varepsilon}(s)\|^2_V ds $ and thus vanishes when applying to  \eqref{originalequation} the It\^o formula for estimating  $\|u^{\varepsilon}(t)\|^2_H$.
\\
\indent Recall that MDP is the large deviation principle (LDP) for the centered process multiplied by a rate slower than that is used in the case of LDP for the original process. In most cases, the method used to prove the LDP for a process can be used to prove its MDP. In the setting of processes formed by solutions to SPDEs, there are two main approaches taken in the literature to establish the LDP. They are the Azencott method introduced by \cite{Azencott, Azencott2} and the weak convergence approach offered by \cite{(11), BDM}. In addition, another technique used by some authors such as those in \cite{Wang2}, specifically to achieve the MDP, is to prove that the process formed for the MDP is exponentially equivalent to the process used to prove the LDP. This is based on \cite[Theorem 4.2.13]{Dembo} which states that if two processes are exponentially equivalent and one satisfies the LDP then the other process also satisfies the LDP. Here we apply the weak convergence approach to prove the MDP for \eqref{originalequation}.\\
\indent In the literature, the LDP for SNLSE  has been established in \cite{(27)} in space $H^{1}(\mathbb{R}^{d})$ by the Azencott method and in \cite{(24)} in space $\mathcal{C}([0,T];L^{2}(0,1))$ by the weak convergence approach. We note that the LPD in \cite{(24)} is on an equation of type \eqref{originalequation}, where the second term is nonlinear, namely $f(u^\varepsilon)=\lambda |u^\varepsilon |^{2\sigma}u^\varepsilon $, which satisfies $\text{Re}(f(u^\varepsilon)-f(u^0),u^\varepsilon-u^0) \ge 0$. In our setting of the MDP, this property does not hold. Hence, we can no longer apply bounds established in \cite{Keller} as in \cite{(24)}   to attain the needed bound for our process in the space $V$. Thus, we have considered  the linear term $f(t,u^\varepsilon)= {\mathcal U}(t) \cdot u^\varepsilon  $. This linear term can be replaced by a nonlinear Lipschitz continuous function $f:[0,T]\times H\to H$ with assumptions similar to those  mentioned in \cite{(29)}.   \\
\indent In general, to achieve the LDP based on the weak convergence approach provided by \cite{BDM}, one forms the skeleton equation (the original equation with noise replaced by a function in $L^2$) and the stochastic controlled equation (the original equation plus the controlled integral from the skeleton equation) and verifies the existence and uniqueness of these two equations along with those of the original equation. The well-posedness of the variational solution of \eqref{originalequation},  its skeleton equation and stochastic controlled equation  follows by using  \cite{(29)} due to similarity in equation (see also \cite{(24)}). For the weak convergence of the stochastic controlled equation to the skeleton equation, as required in this method, we follow the time discretization technique provided by \cite{Bessaih, Millet, Duan}. The main tool for this convergence is to show that the family of solutions to the stochastic controlled equations corresponding to the set $\{h_{\varepsilon}\}_{\varepsilon>0}$ is tight, which enables one to then apply the Skohorod representation theorem and obtain the convergence in distribution required by proving the convergence in probability in the new probability space. Afterwards, by verifying that the limit has the form of the skeleton equation, the condition is attained by noting the uniqueness of solutions.  \\
\indent We also prove a central limit theorem for \eqref{originalequation} by showing that the centered process converges in probability to a unique variational solution of a deterministic PDE. We refer the reader to \cite{Cheng, Hu, Li, Wang, Yang} for other similar results achieved for SPDEs.\\
\indent We begin in Section \ref{sec2} by introducing the notations used throughout the paper and provide the statements of the main results. Sections \ref{sec3} and \ref{sec6} are then devoted to the proof of the MDP by the weak convergence approach and a central limit theorem, respectively. An Appendix is also provided to present some proofs and background material.


\section{Preliminaries} \label{sec2}
\indent In this section we provide the notations and estimates applied in the paper and state the main results.
 {Let $\|\cdot\|$ and $\|\cdot\|_{V}$ denote the norms for $H:= L^{2}(G; {\mathbb C})$ and $V:= H^{1}_0(G; {\mathbb C})$, respectively.}
Assume $(\Omega, \mathcal{F}, (\mathcal{F}_{t})_{t\in [0,T]}, P)$ is a filtered complete probability space. The noise, $W$, is
given as $W(t)= \sum_{j=1}^{\infty} \sqrt{\lambda_{j}}\beta_{j}(t) e_{j}$, where $\{\beta_{j}\}_{j}$ is a sequence of independent one-dimensional real Brownian motions on the  probability space, $\{e_{j}\}_{j}$ is a complete orthonormal system on $H$ acting as eigenvectors corresponding to $\{\lambda_{j}\}_{j}$ with $Qe_{j}=\lambda_{j}e_{j},  {j\in{\mathbb N}}$, where $Q$ is the trace-class covariance operator of the noise. Furthermore, let $H_{0}:=Q^{\frac12}H$ with norm denoted as $|\cdot|_{0}$ and inner product,   $(u,v)_{0}=(Q^{-\frac12}u, Q^{-\frac12}v)$  for $u,v\in H_{0}$.  {We let $L_{2}(H_{0},H)$ and $L_{2}(H_{0},V)$ denote    Hilbert-Schmidt spaces} and impose the conditions below on the noise coefficient $g:[0,T]\times H \to L_{2}(H_{0},H)$
\begin{align}
\|g(t,u)\|_{L_{2}(H_{0},H)}^{2} &\leq  k_{1}\hspace{.08cm}(1+\|u\|^{2}), \hspace{.3cm} \text{for all } \hspace{.3cm} u\in H, \label{(2.6)}\\
\|g(t,u)\|^{2}_{L_{2}(H_{0},V)} &\leq   k_{2}\hspace{.08cm}(1+\|u\|_{V}^{2}), \hspace{.3cm} \text{for all} \hspace{.3cm} u\in V,\label{(2.7)}\\
\|g(t,u)-g(t,v)\|_{L_{2}(H_{0},H)}^{2}&\leq  k_{3}\hspace{.08cm}\|u-v\|^{2}, \hspace{.3cm} \text{for all} \hspace{.3cm} u,v\in H,\label{(2.8)}\\
 {\|g(t,u)-g(t,v)\|_{L_{2}(H_{0},V)}^{2}} & \leq { k_{4}\hspace{.08cm}\|u-v\|^{2}_V, \hspace{.3cm} \text{for all} \hspace{.3cm} u,v\in V,} \label{K4}\\
\|g(t_{1},u)-g(t_{2},u)\|^{2}_{L_{2}(H_{0},H)} &\leq  k_{5}\hspace{.08cm}\|t_{1}-t_{2}\|^{2} \hspace{.3cm} \text{for all} \hspace{.3cm}u\in H, \hspace{.08cm} t_{1}, t_{2}\in [0,T],\label{Holder}
\end{align}
 {where $k_{i}>0$, for $i\in\{1,..,5\}$, is a constant.}
The linear term includes $\mathcal{U}:[0,T]\times G\to \mathbb{C}$,    {for which we assume $\mathcal{U}, \frac{\partial\mathcal{U} }{\partial x_{j}}\in L^\infty([0,T]\times G) , j\in\{1,\ldots, d\},$ and there exists $ k_{6}>0 $ such that
\begin{equation} \label{K5}
|\mathcal{U}(t,x)|^{2} + \sum_{j=1}^{d} \left|\frac{\partial \mathcal{U}(t,x)}{\partial x_{j}}\right|^{2} \leq  {k_{6}}, \hspace{.7cm} \text{for every } (t,x)\in [0,T]\times G.
\end{equation}}
 {The operator $A =-\Delta : V\to V^*$ is defined by}
\begin{equation} \label{eqA}
\langle Au,v\rangle = \int_G \nabla u(x)\nabla\overline{v}(x)dx, \mbox{ for each } u,v\in V  ,
\end{equation}
where $\overline{v}$ is the complex conjugate of $v$,  $V^*$ denotes the dual space of $V$ and $\langle\cdot,\cdot\rangle$ is the duality pairing
of $V$ and $V^*$. In fact, we have $A=-\Delta$ (the Laplacian) with zero Dirichlet boundary conditions. Furthermore, $(\cdot ,\cdot)$ denotes the inner product in $H$.
We consider the variational solution, $u^{\varepsilon}(\cdot)$, of \eqref{originalequation}, i.e. for all $t\in[0,T]$, $v\in V$ and almost all $\omega\in \Omega$,
\begin{equation}\label{(2.2)}
(u^{\varepsilon}(t),v)= (\gamma,v) - i\hspace{.08cm}\!\! \int_{0}^{t} \langle Au^{\varepsilon}(s),v\rangle ds + \int_{0}^{t} \!\!(\mathcal{U}(s)u^{\varepsilon}(s),v)ds + \sqrt{\varepsilon} \int_{0}^{t}\!\! (g(s,u^{\varepsilon}(s))dW(s),v).
\end{equation}
 As mentioned in the introduction, the existence and uniqueness of the solution of \eqref{(2.2)} in the space $$L^{2}(\Omega;\mathcal{C}([0,T];H))\cap L^{2}(\Omega \times [0,T];V)$$follows by    \cite[Theorem 1]{(29)}, where the drift term  in our case is linear. Furthermore, the  following estimates hold for positive constants $ N_{2p}(T)$ and  $N(T)$
\begin{align}
\mathbb{E}\sup_{0\leq t\leq T} \|u^{\varepsilon}(t)\|^{2p} &\leq   N_{2p}(T)\hspace{.08cm} (\varepsilon +\|\gamma\|^{2p}), \mbox{ for any }  {1\leq p<\infty},\label{uHbound}
\end{align}
 and
\begin{align}
  \mathbb{E}\int_0^{T} \|u^{\varepsilon}(t)\|^{2}_Vdt\le  N(T)  (\varepsilon +\|\gamma\|^{2}_V) ,\label{uVbound}
\end{align}
see    \cite[Theorem 2]{(29)} and inequality (18) in the proof of \cite[Theorem 1]{(29)}.   \\
\indent For $\varepsilon=0$, \eqref{(2.2)} becomes a deterministic equation with $u^0\in\mathcal{C}([0,T];H) \cap L^{2}(  0,T  ;V)$ as its unique solution, satisfying
\begin{align}
 \sup_{0\leq t\leq T} \|u^{0}(t)\|^{2 }  \leq   N_{2}(T)  \|\gamma\|^{2 },    \label{u0}
\end{align}
and
\begin{align}
\int_0^T \|u^{0}(t)\|^{2}_V dt \leq     N(T)  \|\gamma\|^{2}_V , \mbox{ for any }  {1\leq p<\infty}.\label{uV0}
\end{align}
In   Subsection  \ref{5.3}  of the Appendix we prove that there exists a   constant $K(T)>0$ such that
\begin{align}
 \sup_{0\leq t\leq T} \|u^{0}(t)\|_{V}^{2 }  \leq   K(T)  \|\gamma\|^{2 }_V.
\label{u0sup}
\end{align}
 \indent Recall that to establish the MDP for the family $\{u^{\varepsilon} \}_{\varepsilon>0}$, the  {LDP} needs to be achieved for $\{v^{\varepsilon} \}_{\varepsilon>0}$, where  {for $a(\varepsilon)>0$}
\begin{equation}\label{(2.1)}
v^{\varepsilon}(t) = \frac{a(\varepsilon)}{\sqrt{\varepsilon}}(u^{\varepsilon}(t)-u^{0}(t)), \hspace{.25cm}\text{with} \hspace{.25cm}  {a(\varepsilon)\rightarrow 0}, \hspace{.25cm} \text{and} \hspace{.25cm} \frac{a(\varepsilon)}{\sqrt{\varepsilon}}\rightarrow \infty \hspace{.25cm} \text{as} \hspace{.25cm} \varepsilon \rightarrow 0.
\end{equation}
 If not otherwise stated, we assume that
\begin{equation} \label{epsi}
 {\varepsilon, a(\varepsilon) \in (0,1)} \hspace{.25cm} \mbox{ and }\hspace{.25cm} \frac{\varepsilon}{a^2(\varepsilon)}<1.
\end{equation}
For  {$h\in L^{2}(0,T;H_{0})$, we let $X^{h}$ be the unique variational solution of the skeleton equation of $v^{\varepsilon}$, given by
\begin{equation}\label{Xhcontrolled}
X^{h}(t)=-i \hspace{.08cm}\int_{0}^{t}\!\! AX^{h}(s) ds + \int_{0}^{t}\!\!\mathcal{U}(s) X^{h}(s)ds+ \int_{0}^{t}\!\! g(s,u^{0}(s))h(s)ds, \hspace{.4cm}t\in[0,T],
\end{equation}
and prove the following main result.}

\begin{theorem} \label{TH1}
Suppose conditions \eqref{(2.6)}-\eqref{K5} hold and let $\varepsilon_{0}\in (0,1)$. Then the family of solutions  $\{u^{\varepsilon} \}_{ \varepsilon\in(0,\varepsilon_{0})}$ satisfies the moderate deviation principle in $\mathcal{C}([0,T];H)$ with rate function,
\begin{align}\label{rate}
I(v) =
\displaystyle \inf\left\{ {\frac{1}{2}}\int_{0}^{T} |h(s)|_{0}^{2}ds: h\in L^2( { 0,T} ;H_0)   \mbox{ with } v= \mathcal{G}\left(\int_{0}^{\cdot}h(s)ds\right)\right\},
\end{align}
 {where the infimum of the empty set is taken to be infinity and} $\mathcal{G}(\int_{0}^{\cdot}h(s)ds)$ is the unique variational solution to the skeleton equation \eqref{Xhcontrolled} corresponding to $h\in L^{2}({0,T};H_{0})$.
\end{theorem}
In addition, considering the process
\begin{equation}\label{Zhat}
\hat{Z}^{\varepsilon}(t) = \frac{u^{\varepsilon}(t)-u^{0}(t)}{\sqrt{\varepsilon}},
\end{equation}
we prove the following central limit theorem for $\{u^{\varepsilon}\}_{\varepsilon \in (0,\varepsilon_{0})}$\, .
\begin{theorem}\label{TH3}
Assuming conditions \eqref{(2.6)}-\eqref{K5}, there exists an $\varepsilon_{0}\in (0,1)$, such that the process, $\{\hat{Z}^{\varepsilon}\}_{\varepsilon\in(0,\varepsilon_{0})}$ converges in probability in space $\mathcal{C}([0,T];H)$, as $\varepsilon$ tends to zero, to  $Z^{0}$, the unique variational solution of
\begin{equation}\label{Vequation}
Z^0(t)= - i \hspace{.08cm}\int_{0}^{t} AZ^0(s)ds + \int_{0}^{t} \mathcal{U}(s) Z^0(s)ds+ \int_{0}^{t} g(s,u^{0}(s))dW(s),
\end{equation}
for all $t\in[0,T]$ and a.e. $\omega\in \Omega$.
\end{theorem}

\section{Moderate Deviations Principle}\label{sec3}

\indent For the weak convergence approach, we let  {$\varepsilon_{0}\in(0,1)$} and consider the process $\{v^{\varepsilon}\}_{\varepsilon \in (0,\varepsilon_{0})}$, defined by \eqref{(2.1)}, where $v^{\varepsilon}$ is the variational solution of
\begin{equation}\label{MDPequation}
v^{\varepsilon}(t) = -i \hspace{.08cm}\int_{0}^{t}  Av^{\varepsilon}(s) ds + \int_{0}^{t} \mathcal{U}(s)v^{\varepsilon}(s) ds + a(\varepsilon) \int_{0}^{t} \hat{g}(s,v^{\varepsilon}(s))dW(s),
\end{equation}
for $t\in[0,T]$, a.e. $\omega\in \Omega$ and  $\hat{g}:[0,T]\times H \to L_2(H_0,H)$ defined by
\begin{equation}
\hat{g}(t, {v} ):= g\left(t, \frac{\sqrt{\varepsilon}}{a(\varepsilon)} {v}
+  u^{0}(t)\right), \ \mbox{ for }  {t\in[0,T], v\in H}.\label{MDPg}
\end{equation}
Following \cite{BDM}, we define for   each $M\ge 0$, 
$$S_{M}:= \left\{h \in L^{2}(0,T;H_{0}): \int_{0}^{T} |h(s)|_{0}^{2}ds \leq M\right\},$$
which under the weak topology, using the metric,
\begin{equation*}
 {d({h_{1},h_{2}})= \sum_{i=1}^{\infty} \frac{1}{2^{i}} \left|\int_{0}^{T} \left( {h_{1}}(s)-{h_{2}}(s), e_{i}(s)\right)_{0}ds\right|, \mbox{ for }h_{1},h_{2}\in S_{M},}
\end{equation*}
 is a compact metric space (see Section 4 of \cite{BDM} or the Appendix in \cite{Budhiraja1}). Moreover, we let 
 \begin{align*}
\mathcal{P}_{2}  :=  \left\{{h:(0,T)\times \Omega\to H_{0}, \, \text{$(\mathcal{F}_{t})_{t\in [0,T]}$-predictable process}:} \int_{0}^{T} |h(s)|_{0}^{2}ds <\infty\  \text{ $P$-a.s.}\right\},
\end{align*}
\begin{align*}
\mathcal{P}_{M} :=  \left\{h\in \mathcal{P}_{2}: h(\omega)\in S_{M}  \ \text{ $P$-a.s.}\right\}.
\end{align*}
Then for $h\in \mathcal{P}_{M}$ we let $v_{h}^{\varepsilon}$ be the variational solution of the following stochastic controlled equation
\begin{eqnarray}
v_{h}^{\varepsilon}(t)  &=&  - i \hspace{.08cm}\int_{0}^{t} Av_{h}^{\varepsilon}(s) ds +  \int_{0}^{t} \mathcal{U}(s)v_{h}^{\varepsilon}(s)ds   + a(\varepsilon) \int_{0}^{t} \hat{g}(s,v_{h}^{\varepsilon}(s))dW(s)\nonumber\\
&& + \int_{0}^{t} \hat{g}(s,v_{h}^{\varepsilon}(s))h(s)ds,  \label{MDPstochastic}
\end{eqnarray}
for $ t\in[0,T]$ and a.e. $\omega\in \Omega$.
The existence and uniqueness of the solutions
\begin{equation}
v^{\varepsilon} ,\ v_{h}^{\varepsilon}\in L^{2}(\Omega;\mathcal{C}([0,T];H))\cap L^{2}(\Omega \times [0,T];V) \label{space}
\end{equation}
of \eqref{MDPequation} and \eqref{MDPstochastic}, respectively, can be proved similar to \cite[Theorem 1]{(29)} (see also the proof of \cite[Theorem 2.1]{(24)}). Furthermore, classical deterministic theory of Schr\"odinger equations may be applied to obtain the well-posedness of solutions to \eqref{Xhcontrolled}, where $X^{h}\in \mathcal{C}([0,T];H) \cap L^{2} {(0,T;V)}$. Note that, in the case  when $h \in \mathcal{P}_{M}$, the properties of $X^h$ hold {$P$-a.s.} We derive for {$1\leq p<\infty$} the following bounds in {Subsections \ref{5.4}, \ref{5.5}  and \ref{5.6} of} the Appendix:
\begin{equation}
\mathbb{E} \sup_{0\leq t\leq T} \|v^{\varepsilon}(t)\|^{2p}\leq a^{2}(\varepsilon)   {\bar N}_{2p}(T),\label{vk}
\end{equation}
\begin{equation}
\mathbb{E} \sup_{0\leq t\leq T} \|v_{h}^{\varepsilon}(t)\|^{2p}\leq \widetilde{N}_{2p}(T, M),\label{vh2pinequality}
\end{equation}
\begin{equation}
\sup_{0\leq t\leq T} \|X^{h}(t)\|^{2p}\leq \hat{N}_{2p}(T, M), \label{Xhinequality}
\end{equation}
{where $\bar N_{2p}(T)$, $\widetilde{N}_{2p} (T, M)$ and $\hat{N}_{2p}(T, M)$ are positive constants defined in \eqref{veps},
 \eqref{A.11} and \eqref{A.12}, respectively.}\\
 \indent {For the following, we denote the identity operator on $H$ by $I$ and let $L_2:=L_2(H_0,H)$.} Within the proofs of our statements, $C_i$, $i\in\{1,2,3,...\}$, represents a generic positive constant, the value of which may vary from line to line.

In the  papers \cite{Bessaih, Millet, Duan} the authors use Galerkin approximations, i.e. solutions of the corresponding stochastic finite dimensional equations, to obtain the estimates which involve the norm in the space $V$, while in our paper we use another method, namely certain Cauchy sequences in the space $L^2(\Omega, C([0,T];V))$ and finite dimensional projections, $\pi_n v_h^\varepsilon$, see   \eqref{proj}, of the  process $v_h^\varepsilon$ solution of \eqref{MDPstochastic} (analogously for $v^{\varepsilon}$ and $X^{h}$). The method is presented in the theorem below, and its result is crucial for the proof of Lemma \ref{tlemma} and  Lemma \ref{LEM2}, which are used to prove the conditions of MDP. We also point out that we consider variational solutions, whereas in \cite{Bessaih, Millet, Duan} weak solutions are considered.

\begin{theorem} \label{TH4}
{Let  $ \varepsilon\in(0,1) $ as in \eqref{epsi} and $h\in \mathcal{P}_{M}$.} The unique variational solutions, $v^{\varepsilon} $, $v_{h}^{\varepsilon},$ and $X^h$ to \eqref{MDPequation}, \eqref{MDPstochastic}, and \eqref{Xhcontrolled}, respectively, satisfy for positive constants $\bar K (T)$, $\widetilde{K} (T, M)$ and $\hat{K} (T,M)$,
\begin{align}
\mathbb{E}\sup_{0\leq t\leq T} \|v^{\varepsilon}(t)\|_{V}^{2} &\leq  \bar K (T), \label{vVnorm}\\
\mathbb{E} \sup_{0\leq t\leq T} \|v^{\varepsilon}_{h}(t)\|_{V}^{2} &\leq  \widetilde{K} (T,M), \label{vhVnorm}\\
\sup_{0\leq t\leq T} \|X^{h}(t)\|_{V}^{2} &\leq  \hat{K} (T,M)\,. \label{XhVnorm}
\end{align}
\end{theorem}
\begin{proof}
We focus on proving \eqref{vhVnorm} and note that the proofs for \eqref{vVnorm} and  \eqref{XhVnorm} are similar. As in \cite[proof of Theorem 1, inequality (18)]{(29)} we can show that there exists  $k(T,M)>0$ such that
\begin{equation}\label{vVM}
 \mathbb{E} \int_{0}^{T}  \|v_{h}^{\varepsilon}(s)\|_{V}^{2}  ds\le   k(T,M).
\end{equation}
For  $n\in{\mathbb N}$ recall that   $\pi_n v_{h}^{\varepsilon}$  is the finite dimensional projection of $v_{h}^{\varepsilon}$, see  \eqref{proj}  in Subsection \ref{5.1} in the Appendix.  \\
\textbf{Step 1:} We prove that \begin{equation}
\lim_{n\rightarrow \infty} \mathbb{E} \sup_{0\leq t\leq T}  \| (I-\pi_{n})v_{h}^{\varepsilon}(t)\|^{2}  =0. \label{conv1}
\end{equation}
The sequence of eigenfunctions of $A$, $\{\varphi_{k}\}_{k\in \mathbb{N}}$, is an orthonormal basis in $H$ (see Subsection \ref{5.1}). By using \eqref{MDPstochastic}  and the It\^o formula we have
\begin{align}\label{new*9}
&|(v_{h}^{\varepsilon}(t), \varphi_{k})|^{2} =    2  \text{Im} \int_{0}^{t}(Av_{h}^{\varepsilon}(s), \varphi_{k})(\overline{v_{h}^{\varepsilon}(s), \varphi_{k}})ds + 2 \text{Re} \int_{0}^{t} \left(\mathcal{U}(s)v_{h}^{\varepsilon}(s), \varphi_{k}\right) (\overline{v_{h}^{\varepsilon}(s),\varphi_{k}})ds \nonumber\\
& + 2a(\varepsilon) \text{Re} \int_{0}^{t} \left(\hat{g}(s,v_{h}^{\varepsilon}(s)) dW(s), \varphi_{k}\right) (\overline{v_{h}^{\varepsilon}(s),\varphi_{k}}) + a^2(\varepsilon) \int_{0}^{t} \sum_{j=1}^\infty |(\hat g(s,v_{h}^{\varepsilon}(s))Q^{\frac12}e_j,\varphi_k )|^2ds\nonumber\\
&+ 2  \text{Re} \int_{0}^{t} \left(\hat{g}(s,v_{h}^{\varepsilon}(s))h(s), \varphi_{k}\right) (\overline{v_{h}^{\varepsilon}(s),\varphi_{k}})ds ,
\end{align}
for all $t\in[0,T]$, and a.e. $\omega\in\Omega$.
We sum up from $k=1$ to $n\in{\mathbb N}$ and then by   \eqref{opA} and \eqref{new*00}  we obtain for each   $n\in \mathbb{N}$,  every $t\in[0,T]$, and a.e. $\omega\in\Omega$
\begin{align} \label{new*2}
&\|\pi_n v_{h}^{\varepsilon}(t)\|^{2} =    2  \text{Re} \int_{0}^{t} \left(\mathcal{U}(s)v_{h}^{\varepsilon}(s), \pi_n v_{h}^{\varepsilon}(s)\right) ds + 2a(\varepsilon) \text{Re} \int_{0}^{t} \left(\hat{g}(s,v_{h}^{\varepsilon}(s)) dW(s), \pi_n v_{h}^{\varepsilon}(s)\right) \nonumber  \\
&+ a^2(\varepsilon) \int_{0}^{t} \sum_{j=1}^\infty  \|\pi_n \hat g(s,v_{h}^{\varepsilon}(s))Q^{\frac12}e_j\|^2ds + 2   \text{Re} \int_{0}^{t} \left(\hat{g}(s,v_{h}^{\varepsilon}(s))h(s), \pi_n v_{h}^{\varepsilon}(s)\right)  ds .
\end{align}
On the other hand, by \eqref{MDPstochastic},  It\^o's formula  and \eqref{opA}  we have for
all $t\in[0,T]$  and a.e. $\omega\in\Omega$
\begin{align} \label{new*3}
&\| v_{h}^{\varepsilon}(t)\|^{2} =    2 \text{Re} \int_{0}^{t} \left(\mathcal{U}(s)v_{h}^{\varepsilon}(s),   v_{h}^{\varepsilon}(s)\right) ds + 2a(\varepsilon) \text{Re} \int_{0}^{t} \left(\hat{g}(s,v_{h}^{\varepsilon}(s)) dW(s),   v_{h}^{\varepsilon}(s)\right) \nonumber  \\
&+ a^2(\varepsilon) \int_{0}^{t} \sum_{j=1}^\infty  \|  \hat g(s,v_{h}^{\varepsilon}(s))Q^{\frac12}e_j\|^2ds + 2  \text{Re} \int_{0}^{t} \left(\hat{g}(s,v_{h}^{\varepsilon}(s))h(s),   v_{h}^{\varepsilon}(s)\right)  ds .
\end{align}
   We substract \eqref{new*2} from \eqref{new*3} and use \eqref{new*1} to obtain
for each   $n\in \mathbb{N}$,  every $t\in[0,T]$, and a.e. $\omega\in\Omega$
\begin{align*} 
& \|(I-\pi_n) v_{h}^{\varepsilon}(t)\|^{2} =     2  \text{Re} \int_{0}^{t} \left( \mathcal{U}(s)v_{h}^{\varepsilon}(s)+\hat{g}(s,v_{h}^{\varepsilon}(s))h(s),(I- \pi_n) v_{h}^{\varepsilon}(s)\right) ds
\nonumber\\& + 2a(\varepsilon) \text{Re} \int_{0}^{t} \left(\hat{g}(s,v_{h}^{\varepsilon}(s)) dW(s), (I-\pi_n) v_{h}^{\varepsilon}(s)\right)
 + a^2(\varepsilon) \int_{0}^{t} \sum_{j=1}^\infty   \|(I-\pi_n)\hat  g(s,v_{h}^{\varepsilon}(s))Q^{\frac12} e_j\|^2ds     .
\end{align*}
Taking the supremum on time, afterwards expectations and using  $ a(\varepsilon)\in(0,1)$, see \eqref{epsi}, we arrive at
\begin{align*}
& \mathbb{E}\sup_{0\leq t\leq T}   \|(I - \pi_{n} ) v_{h}^{\varepsilon}(t)\|^{2} \le
    2\mathbb{E} \int_{0}^{T} \big| (  \mathcal{U}(s) v_{h}^{\varepsilon}(s) + \hat{g}(s, v_{h}^{\varepsilon}(s))h(s) ,
		(I- \pi_{n})   v_{h}^{\varepsilon}(s))\big| ds \\
& +2   \mathbb{E}\sup_{0\leq t\leq T} \Big|\int_{0}^{t}  \left( \hat{g}(s, v_{h}^{\varepsilon}(s))dW, (I-\pi_{n})v_{h}^{\varepsilon}(s)\right)\Big|
 +  \! \mathbb{E} \!\int_{0}^{t} \!    \sum_{j=1}^\infty   \|(I-\pi_n)\hat  g(s,v_{h}^{\varepsilon}(s))Q^{\frac12} e_j\|^2ds.
\end{align*}
Applying the Burkholder-Davis-Gundy inequality and \eqref{new*0}, we obtain
\begin{align*}
&    2\mathbb{E} \sup_{0\leq t\leq T}\Big|   \int_{0}^{t}  \left( \hat{g}(s, v_{h}^{\varepsilon}(s))dW, (I-\pi_{n})v_{h}^{\varepsilon}(s)\right)\Big|\\
&\leq  C_1   \mathbb{E} \Big|\int_{0}^{T}\sum_{j=1}^\infty  \left( (I-\pi_{n}) \hat{g}(s,v_{h}^{\varepsilon}(s))Q^{\frac12}e_j, (I-\pi_{n}) v_{h}^{\varepsilon}(s)\right)^{2}ds\Big|^{\frac12}\\
&\leq  \frac{1}{2}   \mathbb{E} \sup_{0\leq t\leq T}  \left\| (I-  \pi_{n}) v_{h}^{\varepsilon}(t)\right\|^{2} + \frac{C_{1}^2 }{2}    \mathbb{E} \int_{0}^{T}     \sum_{j=1}^\infty   \|(I-\pi_n)\hat  g(s,v_{h}^{\varepsilon}(s))Q^{\frac12} e_j\|^2ds.\nonumber
\end{align*}
Combining the above estimates we have
\begin{flalign}
&\mathbb{E}\sup_{0\leq t\leq T} \|(I-  \pi_{n}) v_{h}^{\varepsilon}(t)\|^{2} \le  C_2\Big(
\mathbb{E} \int_{0}^{T}  \big|  (  \mathcal{U}(s) v_{h}^{\varepsilon}(s)+\hat{g}(s, v_{h}^{\varepsilon}(s)) h(s),(I- \pi_{n})   v_{h}^{\varepsilon}(s) )\big| ds \nonumber \\
&  +    \mathbb{E} \int_{0}^{T}     \sum_{j=1}^\infty   \|(I-\pi_n)\hat  g(s,v_{h}^{\varepsilon}(s))Q^{\frac12} e_j\|^2ds\Big). \label{new}
\end{flalign}
Note that as in \eqref{new*1} it holds
\begin{align*}
  \mathbb{E} \!\!\int_{0}^{T}\!\!  \left\| (I-\pi_{n}) v_{h}^{\varepsilon}(s)\right\|^{2}ds  = \sum_{k= n+1}^\infty\mathbb{E} \!\!\int_{0}^{T}\!\! |(   v_{h}^{\varepsilon}(s),\varphi_k)|^2ds
 \leq  \mathbb{E} \int_{0}^{T}\!\! \| v^{\varepsilon}_{h}(s)\|^{2}ds  .
\end{align*}
By   \eqref{MDPg}, \eqref{new*1},  \eqref{(2.6)}, \eqref{epsi}, \eqref{vh2pinequality}, and \eqref{u0} we get
\begin{align*}
 & \mathbb{E} \int_{0}^{T}     \sum_{j=1}^\infty   \|(I-\pi_n)\hat  g(s,v_{h}^{\varepsilon}(s))Q^{\frac12} e_j\|^2  ds =  \sum_{k= n+1}^\infty	\mathbb{E} \int_{0}^{T}     \sum_{j=1}^\infty    |(  \hat  g(s,v_{h}^{\varepsilon}(s))Q^{\frac12} e_j,\varphi_k)|^2  ds\\
 & \le \mathbb{E} \int_{0}^{T}  \|\hat{g}(s,v_{h}^{\varepsilon}(s))\|^2_{L_2} ds
  \le k_1   \Big(T+\mathbb{E} \int_0^T\! 2\big(    \|v_{h}^{\varepsilon}(s)\|^{2}+   \|u^{0}(s)\|^{2}\big)ds\Big)  .
	\end{align*}
Using the results from Subsection \ref{5.2}  we arrive at
\begin{align}
\lim_{n\rightarrow \infty}   \mathbb{E} \int_{0}^{T} \Big(\|(I- \pi_{n})   v_{h}^{\varepsilon}(s) \|^2+    \sum_{j=1}^\infty   \|(I-\pi_n)\hat  g(s,v_{h}^{\varepsilon}(s))Q^{\frac12} e_j\|^2 \Big)ds=0. \label{new*8}
\end{align}
By the boundedness property of $\mathcal{U} $ given in \eqref{K5} and by   the Cauchy-Schwarz inequality we write
\begin{flalign*}
  \mathbb{E} \int_{0}^{T}\!\!\!  |(  \mathcal{U} (s) v_{h}^{\varepsilon}(s) , (I- \pi_{n})   v_{h}^{\varepsilon}(s))|ds \! \le\! \sqrt{k_6}\Big(  \mathbb{E} \int_0^T \| v^{\varepsilon}_{h}(s)\|^{2} ds \Big)^{\frac12}\!\Big(\mathbb{E} \!\!\int_{0}^{T}\!\!  \left\| (I-\pi_{n}) v_{h}^{\varepsilon}(s)\right\|^{2}ds\Big)^{\frac12}.
\end{flalign*}
 Similarly  we have by using \eqref{MDPg}, \eqref{(2.6)},   \eqref{epsi},   and $h\in \mathcal{P}_{M}$ that
\begin{flalign*}  &   \mathbb{E} \int_{0}^{T}  \big|(\hat{g}(s,v_{h}^{\varepsilon}(s)) h(s) ,(I- \pi_{n})   v_{h}^{\varepsilon}(s))\big| ds \\
&\le \Big(\mathbb{E} \int_{0}^{T}  \|\hat{g}(s,v_{h}^{\varepsilon}(s))\|^2_{L_2}| h(s)|_0^{2}ds\Big)^{\frac12}\Big(\mathbb{E} \!\!\int_{0}^{T}\!\!  \left\| (I-\pi_{n}) v_{h}^{\varepsilon}(s)\right\|^{2}ds\Big)^{\frac12}\\
& \le  \Big(k_1 M \Big(1+2\mathbb{E} \sup_{s\in[{0},{T}]}     \|v_{h}^{\varepsilon}(s)\|^{2}+ 2\sup_{s\in[{0},{T}]}  \|u^{0}(s)\|^{2}\Big)\Big)^{\frac12}\Big(\mathbb{E} \!\!\int_{0}^{T}\!\!  \left\| (I-\pi_{n}) v_{h}^{\varepsilon}(s)\right\|^{2}ds\Big)^{\frac12} .
\end{flalign*}
Thus, by \eqref{new} and the  inequalities, \eqref{vh2pinequality}, \eqref{u0}, and \eqref{new*8}     it follows that
\eqref{conv1} holds.\\
 \textbf{Step 2:} We  show that there exists a positive constant  $\widetilde  K(T,M)$ such that
\begin{equation}\label{import00}
\mathbb{E} \sup_{0\leq t\leq T} \| \pi_{n}v_{h}^{\varepsilon}(t)\|^{2}_V \leq
  \widetilde K(T,M), \mbox{ for each }n\in{\mathbb N}.
\end{equation}
Multiplying both sides of \eqref{new*9} by $\mu_{k}$ (see in Subsection \ref{5.1}  the notation for the eigenvalues of $A$) and summing from $k=1$ to $n$, where $n  \in \mathbb{N}$, we have
\begin{align}
   & \|  \pi_{n} v_{h}^{\varepsilon}(t)\|_V^{2} =   2  \text{Re}   \int_{0}^{t} \left(  \mathcal{U}(s)v_{h}^{\varepsilon}(s)+\hat{g}(s,v_{h}^{\varepsilon}(s))h(s),A   \pi_{n} v_{h}^{\varepsilon}(s) \right)ds \nonumber\\
&+ 2  a(\varepsilon) \text{Re} \int_{0}^{t} \left(  \hat{g}(s,v_{h}^{\varepsilon}(s))dW(s), A   \pi_{n} v_{h}^{\varepsilon}(s)\right)
  +a^2(\varepsilon) \int_{0}^{t} \sum_{j=1}^\infty \|  \pi_n \hat g(s,v_{h}^{\varepsilon}(s))Q^{\frac12}e_j \|^2_Vds. \label{new*13}
	\end{align}
	Denote
\begin{equation} \label{theta0}
\theta(t)=\exp\Big(-\displaystyle  \int_0^t(1+ |h(s) |^{2}_0) ds\Big), \mbox{ for each } t\in[0,T].
\end{equation}
Note that $0<\theta(T)\le \theta(t)\le 1$ for each $t\in[0,T]$. By \eqref{new*13}, \eqref{K5},  \eqref{ineq} and \eqref{epsi} we write
\begin{align*}
   & \theta(t)\|  \pi_{n} v_{h}^{\varepsilon}(t)\|_V^{2}=
		 2 \text{Re}  \!\int_{0}^{t}  \theta(s)( \mathcal{U}(s)v_{h}^{\varepsilon}(s) +\hat{g}(s,v_{h}^{\varepsilon}(s))h(s), A\pi_{n} v_{h}^{\varepsilon}(s) )ds\\
&+ 2  a(\varepsilon) \text{Re}\!\! \int_{0}^{t} \theta(s) \left(  \hat{g}(s,v_{h}^{\varepsilon}(s))dW(s), A   \pi_{n} v_{h}^{\varepsilon}(s)\right)
  +a^2(\varepsilon)\!\! \int_{0}^{t} \!\!\sum_{j=1}^\infty \theta(s) \|  \pi_n \hat g(s,v_{h}^{\varepsilon}(s))Q^{\frac12}e_j \|^2_Vds
	\\&- \int_0^t \theta(s)\big(1+ |h(s) |^{2}_0\big)\|  \pi_{n} v_{h}^{\varepsilon}(s)\|_V^{2}ds\\
& \le \! C_3\int_{0}^{t} \!\! \theta(s)\!\big( \| v_{h}^{\varepsilon}(s)\|^2_V +\|\hat{g}(s,v_{h}^{\varepsilon}(s))\|^2_{L_2(H_0,V)}\big)ds
 + 2  \Big| \!\int_{0}^{t}\!\!\theta(s) \!\left(  \hat{g}(s,v_{h}^{\varepsilon}(s))dW(s), A   \pi_{n} v_{h}^{\varepsilon}(s)\right)\Big|.
	\end{align*}
Using the Burkholder-Davis-Gundy inequality and \eqref{ineq} we compute
\begin{flalign*}
&  2 \mathbb{E} \sup_{0\leq t\leq T}\Big|   \int_{0}^{t} \theta(s)\left( \hat{g}(s, v_{h}^{\varepsilon}(s))dW, A \pi_{n} v_{h}^{\varepsilon}(s)\right)\Big|&\\
&\leq  C_1  \mathbb{E} \Big|\int_{0}^{T}  \sum_{j=1}^\infty \theta^2(s) \left(  \hat{g}(s,v_{h}^{\varepsilon}(s))Q^{\frac12}e_j, A \pi_{n}  v_{h}^{\varepsilon}(s)\right)^{2}ds\Big|^{\frac12}&\\
&\leq  \frac{1}{2}  \mathbb{E} \sup_{0\leq t\leq T}\theta(t) \| \pi_{n} v_{h}^{\varepsilon}(t)\|^{2}_V  +\frac{C_1^2}{2}  \mathbb{E} \int_{0}^{T}
  \theta(s) \|  \hat  g(s,v_{h}^{\varepsilon}(s)) \|^2_{L_2(H_0,V)} ds.
\end{flalign*}
Then we get the estimate
\begin{align*}
&\mathbb{E}\sup_{0\leq t\leq T}\theta(t)  \| \pi_{n}  v_{h}^{\varepsilon}(t)\|^{2}_V \le  C_4 \hspace{.08cm}\mathbb{E} \int_{0}^{T}  \!\! \theta(s)\!\big( \| v_{h}^{\varepsilon}(s)\|^2_V + \|\hat{g}(s,v_{h}^{\varepsilon}(s))\|^2_{L_2(H_0,V)}\big)ds.
\end{align*}
By \eqref{(2.7)} and \eqref{epsi}
\begin{align*}
 \mathbb{E}\sup_{0\leq t\leq T}\theta(t)  \| \pi_{n}  v_{h}^{\varepsilon}(t)\|^{2}_V \le    C_5 \Big( \mathbb{E} \int_{0}^{T} \big(1+\|v_{h}^{\varepsilon}(s)\|_{V}^{2}   +  {\|u^{0}(s)\|^{2}_V}\big)ds\Big)    .
\end{align*}
\eqref{theta0} and $h\in {\mathcal P}_M$ yield
\begin{align*}
&\mathbb{E}\sup_{0\leq t\leq T}  \| \pi_{n}  v_{h}^{\varepsilon}(t)\|^{2}_V \le    C_5\hspace{.08cm} e^{T+ M }\Big( \mathbb{E} \int_{0}^{T} \big(1+\|v_{h}^{\varepsilon}(s)\|_{V}^{2}   + {\|u^{0}(s)\|^{2}_V}\big)ds\Big)    .
\end{align*}
By    \eqref{vVM} and \eqref{uV0}  we have that \eqref{import00} holds with
$$\tilde K(T,M)= C_5\hspace{.08cm} e^{T  +M }\Big(  T +k(T,M)  +   N(T)\|\gamma\|^2_V \Big).$$\\
{\textbf{Step 3:} We  show that $\{\pi_{n}v_{h}^{\varepsilon} \}_{n}$ is a Cauchy sequence in   $L^{2}(\Omega;\mathcal{C}([0,T];V))$.}\\
  By using \eqref{new*13}, \eqref{theta0}, \eqref{new*00} and \eqref{equality}
we have for all $n,m\in{\mathbb N}, n>m$, every $t\in[0,T]$ and a.e. $\omega\in\Omega$
\begin{align*}
   & \theta(t)\|  (\pi_{n}-\pi_{m})v_{h}^{\varepsilon}(t)\|_V^{2} =   2  \text{Re}   \int_{0}^{t} \theta(s)\left(  \mathcal{U}(s)v_{h}^{\varepsilon}(s)+\hat{g}(s,v_{h}^{\varepsilon}(s))h(s),A  (\pi_{n}-\pi_{m})v_{h}^{\varepsilon}(s) \right)ds\\
&+ 2  a(\varepsilon) \text{Re} \int_{0}^{t} \theta(s) \left(  \hat{g}(s,v_{h}^{\varepsilon}(s))dW(s), A  (\pi_{n}-\pi_{m})v_{h}^{\varepsilon}(s)\right)
  \\&+a^2(\varepsilon) \int_{0}^{t} \sum_{j=1}^\infty \theta(s) \| (\pi_n-\pi_m)\hat g(s,v_{h}^{\varepsilon}(s))Q^{\frac12}e_j \|^2_Vds
	\\&- \int_0^t \theta(s) \big(1+  |h(s) |^{2}_0 \big)\| (\pi_n- \pi_m) v_{h}^{\varepsilon}(s)\|_V^{2}ds\\
	&\le  \int_{0}^{t}\| (\pi_{n}-\pi_{m}) \mathcal{U}(s)v_{h}^{\varepsilon}(s)\|^2_Vds+
	  \int_{0}^{t} \|\hat{g}(s,v_{h}^{\varepsilon}(s))-\hat{g}(s,\pi_nv_{h}^{\varepsilon}(s))\|^2_{L_2(H_0,V)}ds\\&+2   \int_{0}^{t}\|\hat{g}(s,\pi_nv_{h}^{\varepsilon}(s))\|_{L_2(H_0,V)} |h(s)|_0\|(\pi_{n}-\pi_{m})v_{h}^{\varepsilon}(s) \|_V ds
	\\&+ 2  a(\varepsilon) \text{Re} \int_{0}^{t} \theta(s) \left(  \hat{g}(s,v_{h}^{\varepsilon}(s))dW(s), A  (\pi_{n}-\pi_{m})v_{h}^{\varepsilon}(s)\right)
  \\&+a^2(\varepsilon) \int_{0}^{t} \sum_{j=1}^\infty \theta(s) \| (\pi_n-\pi_m)\hat g(s,v_{h}^{\varepsilon}(s))Q^{\frac12}e_j \|^2_Vds
	.
\end{align*}
After applying the Burkholder-Davis-Gundy inequality and using \eqref{epsi}, \eqref{K4}, \eqref{(2.7)}, \eqref{theta0} we obtain
\begin{align}
&\mathbb{E}\sup_{0\leq t\leq T} \theta(t)\|(\pi_{n}-\pi_m) v_{h}^{\varepsilon}(t)\|^{2}_V \le C_5\Big(
\mathbb{E} \int_{0}^{T} \| (\pi_{n}-\pi_{m}) \mathcal{U}(s)v_{h}^{\varepsilon}(s)\|^2_Vds
\nonumber \\
&+	 \mathbb{E} \int_{0}^{T} \| (I-  \pi_n)v_{h}^{\varepsilon}(s) \|^2_{ V }ds  + \mathbb{E} \int_{0}^{T}
\sum_{j=1}^\infty   \|(\pi_n-\pi_m)\hat  g(s,v_{h}^{\varepsilon}(s))Q^{\frac12} e_j\|^2_V ds\nonumber \\ & +    \int_{0}^{T}\|\hat{g}(s,\pi_nv_{h}^{\varepsilon}(s))\|_{L_2(H_0,V)} |h(s)|_0\|(\pi_{n}-\pi_{m})v_{h}^{\varepsilon}(s) \|_V ds  \Big). \label{new2}
\end{align}
By the boundedness property \eqref{K5} of $\mathcal{U} $, \eqref{ineq} and by \eqref{vVM}  we have
\begin{align*}
  \mathbb{E} \int_{0}^{T} \left\| (\pi_{n}-\pi_{m}) \mathcal{U}(s)v_{h}^{\varepsilon}(s)\right\|^{2}_Vds \leq  \mathbb{E}\int_{0}^{T} \| \mathcal{U}(s)v_{h}^{\varepsilon}(s)\|^{2}_Vds \leq C_6 k(T,M).
\end{align*}
Applying \eqref{new*1} and \eqref{vVM} we obtain
$$\mathbb{E} \int_{0}^{T} \| (\pi_n-\pi_m)v_{h}^{\varepsilon}(s) \|^2_{ V }ds\le
\mathbb{E} \int_{0}^{T} \| (I-  \pi_n)v_{h}^{\varepsilon}(s) \|^2_{ V }ds\le \mathbb{E} \int_{0}^{T} \|  v_{h}^{\varepsilon}(s) \|^2_{ V }ds\le k(T,M).$$
 By \eqref{(2.7)}, \eqref{epsi}, \eqref{vVM} and \eqref{uV0} we write
\begin{align*}
&\mathbb{E} \int_{0}^{T}
\sum_{j=1}^\infty   \|(\pi_n-\pi_m)\hat  g(s,v_{h}^{\varepsilon}(s))Q^{\frac12} e_j\|^2_V ds\leq \mathbb{E} \int_{0}^{T} \| \hat{g}(s,v_{h}^{\varepsilon}(s))\|_{L_{2}(H_0,V)}^{2}ds \\ &\leq k_2  \mathbb{E} \int_{0}^{T} \Big(1+  2\|v_{h}^{\varepsilon}(s)\|^{2}_V+ 2\|u^{0}(s)\|^{2}_V\Big)ds \le k_2\big(T+2k(T,M)+2N(T)\|\gamma\|^2_V\big)
 .\end{align*}
The properties   for convergent series yield (see Subsection \ref{5.2})
\begin{align*}
&\lim \limits_{\substack{m \to  \infty\\n \to \infty}} \mathbb{E}\!\!\int_{0}^{T}
\!\!\! \big(\|(\pi_{n}-\pi_{m})\mathcal{U}(s)v_{h}^{\varepsilon}(s)\|^{2}_V+  \| (\pi_{n}-\pi_{m})\hat{g}(s,v_{h}^{\varepsilon}(s))\|^{2}_{L_{2}(H_0,V)}\big) ds\! =0,\\
&\lim \limits_{\substack{m \to  \infty\\n \to \infty}} \mathbb{E}\!\!\int_{0}^{T}
\! \| (I-  \pi_n)v_{h}^{\varepsilon}(s) \|^2_{ V }+  \| (\pi_n-  \pi_m)v_{h}^{\varepsilon}(s) \|^2_{ V } \big) ds\! =0.
\end{align*}
By the Cauchy-Schwarz inequality,   \eqref{import00} and \eqref{u0sup} we have for $h\in {\mathcal P}_M$
\begin{align*}
& \mathbb{E}\int_{0}^{T}\|\hat{g}(s,\pi_nv_{h}^{\varepsilon}(s))\|_{L_2(H_0,V)} |h(s)|_0\|(\pi_{n}-\pi_{m})v_{h}^{\varepsilon}(s) \|_{V} ds \\
&\le \Big(\mathbb{E}\int_{0}^{T}\|\hat{g}(s,\pi_nv_{h}^{\varepsilon}(s))\|^2_{L_2(H_0,V)} |h(s)|_0^2ds\Big)^{\frac12}  \Big(\mathbb{E} \int_{0}^{T} \|(\pi_{n}-\pi_{m})v_{h}^{\varepsilon}(s) \|_{V}^2ds \Big)^{\frac12}
\\
&\le \Big(Mk_2\big(1+2\mathbb{E}\sup_{s\in[0,T]}\| \pi_nv_{h}^{\varepsilon}(s) \|^2_{V}+ 2\sup_{s\in[0,T]}\|u^0(s)\|^2_V\big)\Big)^{\frac12}  \Big(\mathbb{E} \int_{0}^{T}\!\! \|(\pi_{n}-\pi_{m})v_{h}^{\varepsilon}(s) \|_{V}^2ds \Big)^{\frac12}\\&
\le \Big(Mk_2\big(1+2\tilde K(T,M)+ 2K(T)\|\gamma\|^2_V\big) \Big)^{\frac12}  \Big(\mathbb{E} \int_{0}^{T} \|(\pi_{n}-\pi_{m})v_{h}^{\varepsilon}(s) \|_{V}^2ds \Big)^{\frac12}.
\end{align*}
Thus, using the above results, \eqref{new2} and the properties of $\theta$  we obtain
\begin{equation}
\lim \limits_{\substack{m \to  \infty\\n \to \infty}} \mathbb{E} \sup_{0\leq t\leq T}  \|(\pi_{n}-\pi_{m})v_{h}^{\varepsilon}(t)\|_{V}^{2}=0 .\label{cauchy}
\end{equation}
 \textbf{Step 4:} We establish \eqref{vhVnorm}.  \\ By \eqref{cauchy} there exists   $\hat{v}_{h}^{\varepsilon} \in L^{2}(\Omega; \mathcal{C}([0,T];V)) $ such that $\{\pi_{n}v_{h}^{\varepsilon} \}_{n}$ converges to
$\hat{v}_{h}^{\varepsilon}$ in this space. We derive
\begin{equation*}
\mathbb{E} \sup_{0\leq t\leq T} \|v^{\varepsilon}_{h}(t)- \hat{v}^{\varepsilon}_{h}(t)\|^{2} \leq 2\hspace{.08cm}\mathbb{E} \sup_{0\leq t\leq T} \|(I-\pi_{n})v_{h}^{\varepsilon}(t)\|^{2} + 2k_0\hspace{.08cm}\mathbb{E} \sup_{0\leq t\leq T} \|\hat{v}^{\varepsilon}_{h}(t)-\pi_{n}v^{\varepsilon}_{h}(t)\|_V^{2},
\end{equation*}
where $k_0$ is the embedding constant of $V\hookrightarrow H$. By \eqref{conv1} the above inequality  leads to
\begin{equation*}
  v_{h}^{\varepsilon}(t) = \hat{v}_{h}^{\varepsilon}(t),   \mbox{ for all }t\in [0,T]\mbox{ and a.e. } \omega\in\Omega.
\end{equation*}
Thus, using \eqref{import00} we obtain
\begin{equation*}
\mathbb{E}\sup_{0\leq t\leq T} \|v_{h}^{\varepsilon}(t)\|_{V}^{2} = \mathbb{E} \sup_{0\leq t\leq T} \|\hat{v}_{h}^{\varepsilon}(t)\|_{V}^{2} =\lim_{n\rightarrow \infty} \mathbb{E}\sup_{0\leq t\leq T} \| \pi_{n}v_{h}^{\varepsilon}(t)\|^{2}_V \leq \widetilde{K} (T,M).
\end{equation*}
\end{proof}
\indent To achieve the MDP, we apply {the following theorem established in \cite{BDM}}, which requires the verification of two conditions.
\begin{theorem} {\normalfont{\cite[Theorem 5]{BDM}}}
{For $\varepsilon>0$ and two Polish spaces, $\mathcal{E}_{0}$ and $\mathcal{E}$,} suppose there exist  measurable maps  ${\mathcal{G}^{\varepsilon}},\mathcal{G}^{0}: \mathcal{E}_{0} \times \mathcal{C}([0,T];H)\rightarrow \mathcal{E}$ satisfying the following two conditions:\\
  1. For every $0<M<\infty$ and compact set {$K\subseteq \mathcal{E}_{0}$},
  \begin{equation*}
  \Gamma_{M,K} = \left\{\mathcal{G}^{0}\left(x, \int_{0}^{\cdot} u(s)ds\right); u\in  {S_{M}} , x\in K\right\},
  \end{equation*}
  is a compact subset of $\mathcal{E}$. \\
  2. Let ${0<M}<\infty$ and $\{h_{\varepsilon}\}_{\varepsilon>0} \subset
{\mathcal{P}_{M}} $ and $\{x^{\varepsilon}\}_{\varepsilon>0}\subset \mathcal{E}_{0}$, where $h_{\varepsilon}\xrightarrow{d} h$ and $x^{\varepsilon}\rightarrow x$ as $\varepsilon\rightarrow 0$, then
  \begin{equation*}
  \mathcal{G}^{\varepsilon} \left(x^{\varepsilon}, \sqrt{\varepsilon} W(\cdot) + \int_{0}^{\cdot} h_{\varepsilon}(s)ds\right) \xrightarrow{d}
  \mathcal{G}^{0}\left(x, \int_{0}^{\cdot} h(s)ds\right).
  \end{equation*}
  Then $\left\{\mathcal{G}^{\varepsilon}(x, \sqrt{\varepsilon}W)\right\}_{\varepsilon>0}$ satisfies the large deviation principle   on $\mathcal{E}$ with rate function,
  \begin{equation*}
  I_{x}(f) = \inf_{\left\{h\in L^{2}([0,T];H_{0}): f= \mathcal{G}^{0}\left(x, \int_{0}^{\cdot}h(s)ds\right)\right\}} \frac{1}{2} \int_{0}^{T} |h(s)|_{0}^{2}ds,
  \end{equation*}
where the infimum of the empty set is taken to be infinity.
    \end{theorem}
  \indent The first condition above, which guarantees that the rate function is a good rate function, is to verify that the set  {$\{X^{h}: h\in S_{M}\}$} is a compact set, where $X^{h}$ is the unique solution to the skeleton equation {\eqref{Xhcontrolled}} corresponding to the controlled function $h$. This condition may be proved by showing that the set,  {$\{X^{h}:h\in S_{M}\}$} is sequentially compact as shown in \cite{Bessaih, Millet, Duan} or to verify that the map $h\mapsto X^{h}$ is continuous in the weak topology, since the set used for function $h$ is a compact set under the weak topology (see for example \cite{Cheng, Hu, Li, Xu}). The second condition is if {$h_{\varepsilon}\in \mathcal{P}_{M}$} converges weakly to $h\in \mathcal{P}_{M}$ in $L^2({0,T}; H_0)$, as $\varepsilon$ tends to zero, then the stochastic controlled equation converges to the skeleton equation in distribution as $\varepsilon$ goes to zero.   We will achieve the second condition in Lemma \ref{LEM2} and use elements of its proof to verify the first condition in Lemma \ref{LEM3}.
	 	 {First we} need the lemma below, for which we introduce for {$n\in \mathbb{N}$} the step function $\psi_n:[0,T]\to[0,T]$ defined by
\begin{equation} \label{psin}
 \psi_n(s)=\frac{kT}{2^n}, \mbox{ if } s\in \left[ \frac{(k-1)T}{2^n},\frac{kT}{2^n}\right),\  k\in\{1,...,2^{n}\} \mbox{ and }\psi_n(T)=T.
\end{equation} Observe that $|\psi_n(s)-s|\le T2^{-n}$ for each $s\in[0,T]$.

\begin{lemma}\label{tlemma}
For $n\in \mathbb{N}$ consider the step function $\psi_n$ as defined in \eqref{psin}. Then the solutions to \eqref{MDPstochastic} and \eqref{Xhcontrolled} satisfy
 \begin{equation}\label{Ctilde}
 \mathbb{E} \int_{0}^{T} \left\|v_{h}^{\varepsilon}(\psi_n(s))-v_{h}^{\varepsilon}(s)\right\|^{2}ds\leq \widetilde{C}(T, M)\hspace{.08cm}  2^{-\frac{n}{2}},
 \end{equation}
and
 \begin{equation}\label{Xh}
 \int_{0}^{T} \big\|X^{h}(\psi_n(s))-X^{h}(s)\big\|^{2}ds\leq \hat{C}(T, M) \hspace{.08cm} 2^{-n},
 \end{equation}
 respectively.
 \end{lemma}
 \begin{proof}
 For better presentation, we let,
$$
 \overline{V}^{\varepsilon}(h_{1}, s_{1};h_{2}, s_{2}):= v_{h_{1}}^{\varepsilon}(s_{1})-v_{h_{2}}^{\varepsilon}(s_{2}).
 $$
We apply   It\^o's formula, then take the integral from $0$ to $T$ and afterwards the expectation to obtain by \eqref{(2.6)}
\begin{align*}
&\mathbb{E} \int_{0}^{T} \|\overline{V}^{\varepsilon}(h,\psi_n(s);h,s)\|^{2}ds \le  2\hspace{.08cm}  \mathbb{E} \int_{0}^{T} \int_{s}^{\psi_n(s)} \text{Im}  \langle  Av_{h}^{\varepsilon}(r), \overline{V}^{\varepsilon}(h,r;h,s)\rangle drds
\\
&+ 2\hspace{.08cm}   \mathbb{E} \int_{0}^{T}\int_{s}^{\psi_n(s)}\!\!\Big|\left(\mathcal{U}(r)v_{h}^{\varepsilon}(r), \overline{V}^{\varepsilon}(h,r;h,s)\right) \Big|drds
\\
&+ 2\hspace{.08cm}a(\varepsilon)\hspace{.08cm}    \mathbb{E} \Big| \int_{0}^{T}\int_{s}^{\psi_n(s)} \!\!\left(\hat{g}(r,v_{h}^{\varepsilon}(r))dW(r), \overline{V}^{\varepsilon}(h,r;h,s)\right)ds\Big|
\\
&+ a^2(\varepsilon)  \hspace{.08cm}  \mathbb{E} \int_{0}^{T}\int_{s}^{\psi_n(s)}\left\|\hat{g}(r,v_{h}^{\varepsilon}(r))\right\|_{L_{2}}^{2}drds&\\
&+ 2  \hspace{.08cm}  \mathbb{E} \int_{0}^{T}\int_{s}^{\psi_n(s)} \Big|\left(\hat{g}(r,v_{h}^{\varepsilon}(r))h(r), \overline{V}^{\varepsilon}(h,r;h,s)\right)\Big|drds
\\
 &= J_{1}+J_{2}+J_{3}+J_{4}+J_{5} .
\end{align*}
By the properties of $A$, estimate \eqref{vhVnorm} and the definition of   $\psi_n $, it follows
 \begin{align*}
&J_1\leq    2 \hspace{.08cm} \mathbb{E}\int_{0}^{T}\int_{s}^{\psi_n(s)}\!\!  \|v_{h}^{\varepsilon}(r)\|_V \|\overline{V}^{\varepsilon}(h,r;h,s)\|_V  drds\\
 &\le 3\hspace{.08cm} \mathbb{E}\int_{0}^{T} \int_{s}^{\psi_n(s)}\!\! \|v_{h}^{\varepsilon}(r)\|_{V}^{2}drds +  2 \mathbb{E}\int_{0}^{T} \int_{s}^{\psi_n(s)} \!\!\|v_{h}^{\varepsilon}(s)\|_{V}^{2}drds\leq 5\hspace{.08cm}\widetilde{K} (T,M) \hspace{.08cm}T^2\hspace{.08cm} 2^{-n}.
\end{align*}
By \eqref{K5} and \eqref{vh2pinequality} we have
\begin{align*}
&J_2\le   {k_6}\hspace{.08cm} \mathbb{E} \int_{0}^{T}\int_{s}^{\psi_n(s)} \!\!   \|v_{h}^{\varepsilon}(r)\|^{2}drds
 +  \mathbb{E}\int_{0}^{T} \int_{s}^{\psi_n(s)} \!\!\|\overline{V}^{\varepsilon}(h,r;h,s)\|^{2}drds \\
& \le  ({k_6}+2) \hspace{.08cm}\mathbb{E} \int_{0}^{T}\int_{s}^{\psi_n(s)}  \!\!  \|v_{h}^{\varepsilon}(r)\|^{2}drds+   2\hspace{.08cm}\mathbb{E}\int_{0}^{T} \int_{s}^{\psi_n(s)} \!\!\|v_{h}^{\varepsilon}(s)\|^{2}drds\\& \le  (k_6+4)\hspace{.08cm} \widetilde{N}_2 (T,M) \hspace{.08cm}T^2 \hspace{.08cm}2^{-n} .
\end{align*}
By the  Burkholder-Davis-Gundy inequality, then by the Cauchy-Schwarz inequality, \eqref{(2.6)} and \eqref{vh2pinequality} we obtain
\begin{align*}
 J_3 &\le C_1 \int_{0}^{T}  \mathbb{E} \Big( \int_{s}^{\psi_n(s)} \!\!\left(\hat{g}(r,v_{h}^{\varepsilon}(r))  , \overline{V}^{\varepsilon}(h,r;h,s)\right)^2dr\Big)^{\frac12}ds\\
\le & C_1\Big(T\hspace{.08cm} \mathbb{E}\!\! \sup_{s\in[0,T]}\|\hat{g}(s,v_{h}^{\varepsilon}(r))\|^2_{L_2}\Big)^{\frac12} \!  \Big(2\hspace{.08cm}  \mathbb{E} \!\!\int_{0}^{T}\!\!\int_{s}^{\psi_n(s)} \!\!   \|v_{h}^{\varepsilon}(r)\|^{2}  +\|v_{h}^{\varepsilon}(s)\|^{2}drds \Big)^{\frac12}\\
&\le  2\hspace{.08cm}C_1\hspace{.08cm}T\hspace{.08cm}2^{-\frac{n}{2}}\left( T  \hspace{.08cm} k_1 \big( 1+2\hspace{.08cm}
\widetilde{N}_2 (T,M)  +2\hspace{.08cm}N_2(T)\|\gamma\|^2\big)  \widetilde{N}_2 (T,M)
\right)^{\frac12} .
 \end{align*}
By \eqref{(2.6)},   \eqref{vh2pinequality} and \eqref{u0}
\begin{align*}
J_4 &\le   a^2(\varepsilon)\hspace{.08cm} k_1   \mathbb{E}\int_{0}^{T}\int_{s}^{\psi_n(s)}\Big(1+ \Big\|\frac{\sqrt{\varepsilon}}{a(\varepsilon)} v_{h}^{\varepsilon}(r)+ u^{0}(r)\Big\|^{2}\Big)drds \\
& \le \left( 1+2  \widetilde{N}_2 (T,M)  +2 N_2(T) \|\gamma\|^2\right) k_1  T^2  2^{-n}\,.
\end{align*}
By the Cauchy-Schwarz inequality, \eqref{(2.6)},   \eqref{vh2pinequality} and \eqref{uHbound} (for $\varepsilon=0$) we have
\begin{align*}
&J_5 \le   2 \hspace{.08cm}\mathbb{E}\int_{0}^{T}\int_{s}^{\psi_n(s)} |h(r)|_{0} \|\hat{g}(r,v_{h}^{\varepsilon}(r))\|_{L_2} \|\overline{V}^{\varepsilon}(h,r;h,s)\| drds\\
& \le 2\Big(\mathbb{E}\!\!\int_{0}^{T}\!\!\int_{s}^{\psi_n(s)} |h(r)|_{0}^2 \|\hat{g}(r,v_{h}^{\varepsilon}(r))\|^2_{L_2}  drds\Big)^{\frac12}\!\Big(\mathbb{E}\!\!\int_{0}^{T}\!\!\int_{s}^{\psi_n(s)}  \|\overline{V}^{\varepsilon}(h,r;h,s)\|^2 drds \Big)^{\frac12}\\
& \le 2\Big(T M\mathbb{E} \sup_{s\in[0,T]}\|\hat{g}(s,v_{h}^{\varepsilon}(s))\|^2_{L_2}\Big)^{\frac12}  \Big(2 \mathbb{E} \!\!\int_{0}^{T}\!\!\int_{s}^{\psi_n(s)} \!\!   \|v_{h}^{\varepsilon}(r)\|^{2}+ \|v_{h}^{\varepsilon}(s)\|^{2}drds \Big)^{\frac12}\\
&\le  4\hspace{.08cm}T\hspace{.08cm}2^{-\frac{n}{2}}\left( T M  k_1 \big( 1+2
\widetilde{N}_2 (T,M)  +2\hspace{.08cm}N_2(T)\|\gamma\|^2\big)  \widetilde{N}_2 (T,M)
\right)^{\frac12}.
 \end{align*}
The proof of \eqref{Xh} uses similar ideas, where the estimates are pathwise.
\end{proof}
\begin{lemma} \label{LEM2}
Suppose $h\in \mathcal{P}_{M}$, $\varepsilon_{0}\in(0,1)$ and let $v^{\varepsilon}_{h} $ and $X^{h} $ be the unique solutions to \eqref{MDPstochastic} and \eqref{Xhcontrolled}, respectively. Then for a family $\{h_{\varepsilon} \}_{ \varepsilon\in(0,\varepsilon_{0})}$ in $\mathcal{P}_{M}$ converging {$P$-}a.s. weakly to $h\in \mathcal{P}_{M}$ in the weak topology of $L^2( {0,T}; H_0)$, $\{v^{\varepsilon}_{h_{\varepsilon}}\}_{\varepsilon\in(0,\varepsilon_{0})} $ converges in distribution to $X^{h} $  in the space $\mathcal{C}([0,T];H)$,   as $\varepsilon$ tends to zero.
\end{lemma}
\begin{proof}
Recall  the Sobolev space, $W^{\alpha, p}(0,T;\mathbb{H})$ on a separable Hilbert space, $\mathbb{H}$, with $\alpha \in (0,1)$, and  {$1<p<\infty$}, consisting of all maps  $v\in L^{p}(0,T;\mathbb{H})$ satisfying
\begin{equation*}
\int_{0}^{T}\int_{0}^{T} \frac{\|v(t)-v(s)\|_{\mathbb{H}}^{p}}{|t-s|^{1+\alpha p}}dtds<\infty,
\end{equation*}
with norm defined by
\begin{equation}\label{(4.1)}
\|v\|^{p}_{W^{\alpha,p}(0,T;\mathbb{H})}:= \int_{0}^{T} \|v(t)\|_{\mathbb{H}}^{2}dt + \int_{0}^{T}\int_{0}^{T} \frac{\|v(t)-v(s)\|_{\mathbb{H}}^{p}}{|t-s|^{1+\alpha p}}dt ds.
\end{equation}
 We may show using \eqref{vhVnorm} that $v_{h}^{\varepsilon}$ is bounded in probability in $W^{\alpha,2}(0,T;V)$ and the compact embedding of $W^{\alpha,2}(0,T;V)\subset \mathcal{C}([0,T];H)$ for $\alpha\in (\frac12, 1)$, achieved in \cite[Theorem 2.2]{(25)}, may be applied to obtain the tightness of the family $\{v^{\varepsilon}_{h_{\varepsilon}}\}_{\varepsilon \in (0,\varepsilon_{0})}$ in $\mathcal{C}([0,T];H)$. Hence, there is a subsequence that we still denote as $\{v^{\varepsilon}_{h_{\varepsilon}}\}_{\varepsilon \in (0,\varepsilon_{0})}$
that converges in distribution to some process {$v_{h}$} as $\varepsilon$ tends to zero. Applying the  {Skorohod} representation theorem,
 we obtain another family  {$\{\widetilde{v}^{\varepsilon}_{\widetilde h_{\varepsilon}}\}_{\varepsilon \in (0,\varepsilon_{0})}$ and a map $\widetilde{v}_{\widetilde h}$
on a stochastic basis
$(\widetilde{\Omega}, \widetilde{\mathcal{F}}, (\widetilde{\mathcal{F}}_{t})_{t\in[0,T]},\widetilde{P})$ such that   $(\widetilde h_{\varepsilon},\widetilde{W})$  has the same joint distribution as $(h_{\varepsilon}, W)$,  $\widetilde h \overset{d}{=} h$,
 $\widetilde{v}^{\varepsilon}_{\widetilde h_{\varepsilon}}\overset{d}{=} v^{\varepsilon}_{h_{\varepsilon}}$ and $\widetilde{v}_{\widetilde h} \overset{d}{=} v_{h}$ with $\widetilde{v}_{\widetilde h_{\varepsilon}}^{\varepsilon} \to \widetilde{v}_{\widetilde h}$   in $C([0,T];H)$ $\widetilde{P}$-a.s. and $\widetilde h_\varepsilon\to \widetilde h$  weakly in $\mathcal{P}_{M}$ $\widetilde{P}$-a.s.
}

 Because of the equivalence in law, in order to show the convergence in probability of $$\sup_{0\leq t\leq T}\| {v}_{h_{\varepsilon}}^{\varepsilon}(t)-X^{h}(t)\|^2\rightarrow 0, \mbox{ as  }\varepsilon \rightarrow 0,$$   it is sufficient to prove that
 {\begin{equation}\label{expectationlimit}
 \widetilde{\mathbb{E}} \sup_{0\leq t\leq T}\|\widetilde{v}_{\widetilde  h_{\varepsilon}}^{\varepsilon}(t)-\widetilde{v}_{\widetilde  h}(t)\|^2\rightarrow 0,   \mbox{ as  }\varepsilon \rightarrow 0,
 \end{equation}
 where $\widetilde{v}_{\widetilde  h}$ is the variational solution of
$$\widetilde{v}_{\widetilde  h}(t)= -i \hspace{.08cm}\int_{0}^{t} A\widetilde{v}_{\widetilde  h}(s)ds + \int_{0}^{t} \mathcal{U}(s)\widetilde{v}_{\widetilde  h}(s)ds + \int_{0}^{t} g(s,u^{0}(s))\widetilde h(s)ds,
$$}
and then use the uniqueness of solution  of \eqref{Xhcontrolled}, to obtain the limit $X^{h}$.
 {For simplicity we write further in this proof  $h_\varepsilon$ instead of $\widetilde h_\varepsilon$ and $h $ instead of $\widetilde h $ for each  $\widetilde h_\varepsilon, \widetilde h \in \mathcal{P}_{M}$.}
Applying the notation
\begin{eqnarray*}
 \widetilde{V}^{\varepsilon}(h_{1}, s_{1}; h_{2}, s_{2}):= \widetilde{v}_{h_{1}}^{\varepsilon}(s_{1})-\widetilde{v}_{h_{2}}(s_{2}),
\end{eqnarray*}
 {for $h_{1}, h_{2} \in \mathcal{P}_{M}$ and $s_{1}, s_{2} \in [0,T]$,} we have by It\^o's formula,
\begin{flalign*}
&\|\widetilde{V}^{\varepsilon}(h_{\varepsilon}, t; h, t)\|^{2} \leq
 2 \hspace{.08cm} \text{Re} \int_{0}^{t}  \left(\mathcal{U}(s) \widetilde{V}^{\varepsilon}(h_{\varepsilon}, s; h, s), \widetilde{V}^{\varepsilon}(h_{\varepsilon}, s; h, s)\right) ds&\\
& + 2\hspace{.08cm} a(\varepsilon)\hspace{.08cm}\text{Re}\int_{0}^{t}  \left(\hat{g}(s,\widetilde{v}_{h_{\varepsilon}}^{\varepsilon}(s))d{\widetilde{W}}(s), \widetilde{V}^{\varepsilon}(h_{\varepsilon}, s; h, s)\right) + a^{2}(\varepsilon)  \int_{0}^{t} \|\hat{g}(s,\widetilde{v}^{\varepsilon}_{h_{\varepsilon}}(s))\|_{L_{2}}^{2}ds &\\
&+ 2\hspace{.08cm}\text{Re} \int_{0}^{t}\left(\left(\hat{g}(s,\widetilde{v}_{h_{\varepsilon}}^{\varepsilon}(s))h_{\varepsilon}(s)- g(s,u^{0}(s))h(s)\right),  \widetilde{V}^{\varepsilon}(h_{\varepsilon}, s; h, s)\right)ds.&
\end{flalign*}
Taking the supremum on time and then expectation, we have using the Burkholder-Davis-Gundy inequality, \eqref{K5} and \eqref{(2.6)}
\begin{flalign*}
&\widetilde{\mathbb{E}}\sup_{0\leq t\leq T} \|\widetilde{V}^{\varepsilon}(h_{\varepsilon}, t; h, t)\|^{2}\leq   2\hspace{.08cm}\sqrt{k_6}\hspace{.08cm}  \hspace{.08cm}\widetilde{\mathbb{E}}\int_{0}^{T}  \| \widetilde{V}^{\varepsilon}(h_{\varepsilon}, s; h, s)\|^{2}ds
&\\
& + \frac{1}{2}\hspace{.08cm}\widetilde{\mathbb{E}} \sup_{0\leq t\leq T} \|\widetilde{V}^{\varepsilon}(h_{\varepsilon}, t; h, t)\|^{2}+ a^{2}(\varepsilon)C_1  \hspace{.08cm} \widetilde{\mathbb{E}} \int_{0}^{T}  \left(1+ \frac{\varepsilon}{a^{2}(\varepsilon)} \|\widetilde{v}^{\varepsilon}_{h_{\varepsilon}} (s)\|^{2} + \|u^{0}(s)\|^{2}\right)ds&\\
&+ 2\hspace{.08cm} \widetilde{\mathbb{E}}\sup_{0\leq t\leq T} \int_{0}^{t} \text{Re}\left(g(s,u^{0}(s))(h_\varepsilon(s)-h(s)),  \widetilde{V}^{\varepsilon}(h_{\varepsilon}, s; h, s)\right)ds&\\
&+ 2\hspace{.08cm}\widetilde{\mathbb{E}} \int_{0}^{t} \Big| \left(\left(\hat{g}(s,\widetilde{v}_{h_{\varepsilon}}^{\varepsilon}(s))- g(s,u^{0}(s))\right)h_{\varepsilon}(s), \widetilde{V}^{\varepsilon}(h_{\varepsilon}, s; h, s)\right)\Big|ds & \\
&= I_{1}+I_{2}+I_{3}+I_{4}+I_{5}.&
\end{flalign*}
Applying \eqref{vh2pinequality} and \eqref{u0}, we write $$I_{3}\le C_1\hspace{.08cm}T(\varepsilon  \widetilde{N}_2(T,M)  +a^{2}(\varepsilon)(1+N_2(T) {\|\gamma\|^2})).$$
For $I_{5}$, we obtain by \eqref{(2.8)} and \eqref{vh2pinequality},
\begin{eqnarray*}
I_{5}&\leq&  k_{3}\hspace{.08cm}\frac{ {\varepsilon}}{a^2(\varepsilon)} \hspace{.08cm}\widetilde{\mathbb{E}} \sup_{s\in[0,T]}\|\widetilde{v}_{h_{\varepsilon}}^{\varepsilon}(s)\|^{2} \int_{0}^{T} |h^{\varepsilon}(s)|_{0}^{2}ds   + \widetilde{\mathbb{E}} \int_{0}^{T} \|\widetilde{V}^{\varepsilon}(h_{\varepsilon}, s; h, s)\|^{2} ds\\
&\leq&  k_{3}\hspace{.08cm}\frac{\varepsilon}{a^{2}(\varepsilon)}\hspace{.08cm} \widetilde{N}_{2}(T, M) M +   \widetilde{\mathbb{E}} \int_{0}^{T}\|\widetilde{V}^{\varepsilon}(h_{\varepsilon}, s; h, s)\|^{2} ds.
\end{eqnarray*}
Considering the term $I_{4}$, we note that the assumption of $h_{\varepsilon}$ converging weakly to $h$ in the weak topology of $L^2(0,T; H_0)$,   implies that  {$\widetilde{P}$-}a.s. for each $a,b\in[0,T]$, $\int_{a}^{b} h_{\varepsilon}(s)ds$ converges to $\int_{a}^{b}h(s)ds$ in the weak topology of $H_{0}$ as $\varepsilon\to0$. We adopt the technique offered by \cite{Millet, Duan} to estimate $I_{4}$ by applying a time discretization. For   {$n\in{\mathbb N}$} consider the partition of $[0,T]$: $ [t_{k-1},t_{k}):= \left[\frac{(k-1)T}{2^{n}}, \frac{kT}{2^{n}}\right)$  for $k\in \{1,..., 2^{n}\}$, as well as $\psi_n$ defined in \eqref{psin}.
We write
\begin{align*}
&I_{4}= 2\hspace{.08cm}  \widetilde{\mathbb{E}} \sup_{0\leq t\leq T}\int_{0}^{t}\text{Re} \left(g(s,u^{0}(s))(h_\varepsilon(s)-h(s)),  \widetilde{V}^{\varepsilon}(h_{\varepsilon}, s; h, s)\right)ds   \\
&\leq 2 \hspace{.08cm} \widetilde{\mathbb{E}} \int_{0}^{T} \left|\left(g(s,u^{0}(s))(h_\varepsilon(s)-h(s)), \widetilde{V}^{\varepsilon}(h_{\varepsilon},s; h,s)-  \widetilde{V}^{\varepsilon}(h_{\varepsilon}, \psi_n(s);h,\psi_n(s))\right)\right|ds
\\
&+ 2  \hspace{.08cm}\widetilde{\mathbb{E}} \int_{0}^{T} \left| \left(\big(g(s,u^{0}(s))-g( \psi_n(s), u^{0}(s))\big)(h_\varepsilon(s)-h(s)), \widetilde{V}^{\varepsilon}(h_{\varepsilon},\psi_n(s); h, \psi_n(s))\right)\right|ds \\
&+ 2\hspace{.08cm} \widetilde{\mathbb{E}} \sum_{k=1}^{2^{n}} \left|\Big(g(t_{k},u^{0}(t_{k}))\int_{t_{k-1}}^{t_{k}} (h_\varepsilon(s)-h(s)) ds,  \widetilde{V}^{\varepsilon}(h_{\varepsilon}, t_{k};h, t_{k})\Big)\right| \\
&+  2\hspace{.08cm} \widetilde{\mathbb{E}} \sup_{1\leq k\leq 2^{n}}\sup_{t_{k-1}\leq t\leq t_{k}} \Big|\Big(g(t_{k},u^{0}(t_{k})) \int_{t_{k-1}}^{t}(h_\varepsilon(s)-h(s))ds, \widetilde{V}^{\varepsilon}(h_{\varepsilon}, t_{k}; h, t_{k})\Big)\Big| \\
&= J_{1}(n,\ve)+J_{2}(n,\ve)+J_{3}(n,\ve)+ J_{4}(n,\ve).
\end{align*}
Using the Cauchy-Schwarz inequality and \eqref{(2.6)}, we find
\begin{eqnarray*}
&&J_{1}(n,\ve)\\
&&\leq 2\hspace{.08cm}\sqrt{2M}  \left(\int_{0}^{T} k_{1}(1+\|u^{0}(s)\|^{2})\hspace{.08cm}\mathbb{E} \| \widetilde{V}^{\varepsilon}(h_{\varepsilon},s; h,s)-  \widetilde{V}^{\varepsilon}(h_{\varepsilon}, \psi_n(s);h,\psi_n(s))\|^{2} ds\right)^{\frac12},
\end{eqnarray*}
and with the observation
\begin{equation*}
\|\widetilde{V}^{\varepsilon}(h_{\varepsilon},s;h,s)-\widetilde{V}^{\varepsilon}(h_{\varepsilon},\psi_n(s);h, \psi_n(s))\|^{2} \leq 2\|\widetilde{v}_{h_{\varepsilon}}^{\varepsilon}(s)- \widetilde{v}_{h_{\varepsilon}}^{\varepsilon}(\psi_n(s))\|^{2} + 2\|\widetilde{v}_{h} (s)- \widetilde{v}_{h}  (\psi_n(s))\|^{2},
\end{equation*}
we obtain by Lemma \ref{tlemma} and  {\eqref{u0}}
\begin{equation*}
J_{1}(n,\ve) \leq    4\left(M k_1(1+N_2(T) {\|\gamma\|^2}) (\widetilde{C}(T,M)+\hat{C}(T,M))\right)^{\frac12}\hspace{.08cm} 2^{-\frac{n}{4}}.
\end{equation*}
By  \eqref{vh2pinequality} and \eqref{Xhinequality}  we obtain
 $$ \widetilde{\mathbb{E}}\int_0^T\| \widetilde{V}^{\varepsilon}(h_{\varepsilon},\psi_n(s); h, \psi_n(s))\|^{2}ds \leq 2\hspace{.08cm} T\big(\widetilde{N}_{2}(T,M)+ \hat{N}_{2}(T,M)\big).$$
Therefore by \eqref{Holder} and $h_\varepsilon, h\in {\mathcal P}_M$  it holds
\begin{equation*}
J_{2}(n,\ve) \leq    2  \sqrt{k_4 \hspace{.08cm}T}  \left(    T\big( \widetilde{N}_{2}(T,M)+ \hat{N}_{2}(T,M)\big) +M  \right)   2^{-\frac{n}{2}}.
\end{equation*}
In addition, for each fixed $k\in \{1,..., 2^{n}\}$, $g(t_{k}, u^{0}(t_{k})) $ being a Hilbert-Schmidt operator implies that it is also a compact operator, hence it is also strongly continuous, and
  \begin{equation*}
\Big\|g(t_{k}, u^{0}(t_{k})) \int_{t_{k-1}}^{t_{k}}( h_\varepsilon(s)-h(s)) ds\Big\|  \to 0 \mbox{ as }  \varepsilon\to 0 , \  {\widetilde{P}}\mbox{-a.s.}
\end{equation*}
The dominated convergence theorem implies
\begin{equation*}
\widetilde{\mathbb{E}} \Big\|g(t_{k}, u^{0}(t_{k})) \int_{t_{k-1}}^{t_{k}}( h_\varepsilon(s)-h(s)) ds\Big\|^2  \to 0 \mbox{ as }  \varepsilon\to 0 .
\end{equation*}
Hence
  \begin{align*}
	& \widetilde{\mathbb{E}}  \Big|\Big(g(t_{k},u^{0}(t_{k}))\int_{t_{k-1}}^{t_{k}} (h_\varepsilon(s)-h(s)) ds,  \widetilde{V}^{\varepsilon}(h_{\varepsilon}, t_{k};h, t_{k})\Big)\Big|\\ &\le
	\Big(  \widetilde{\mathbb{E}} \Big\|g(t_{k}, u^{0}(t_{k}))\int_{t_{k-1}}^{t_{k}}(h_\varepsilon(s)-h(s))ds\Big\|^2\Big)^{\frac12}\Big(\widetilde{\mathbb{E}}\| \widetilde{V}^{\varepsilon}(h_{\varepsilon}, t_{k};h, t_{k})\|^2\Big)^{\frac12} \to 0 \mbox{ as }  \varepsilon\to 0.
\end{align*}
For fixed $n$ we then have   $$\lim_{\ve\to 0}J_{3}(n,\ve)=0 .$$
By \eqref{(2.6)}, \eqref{vh2pinequality}, \eqref{Xhinequality} and $h_\varepsilon, h\in {\mathcal P}_M$ we  write
\begin{align*}
&  J_{4}(n,\ve)\leq
  2\hspace{.08cm}\widetilde{\mathbb{E}} \sup_{1\leq k\leq 2^{n}}\|g(t_{k},u^{0}(t_{k})) \|_{L_2}\|\widetilde{V}^{\varepsilon}(h_{\varepsilon}, t_{k}; h, t_{k})\|   \int_{t_{k-1}}^{t_k}|h_\varepsilon(s)-h(s)|_0ds      \\
&  \leq 2\Big( k_{1}(1 +N_2(T) {\|\gamma\|^2})2\big(\widetilde{N}_{2}(T,M)+ \hat{N}_{2}(T,M)\big)\Big)^{\frac12}&\\
&\hspace{6cm} \times \Big(\sup_{1\leq k\leq 2^{n}}(t_k-t_{k-1})\int_{t_{k-1}}^{t_{k}} | h_\varepsilon(s)-h(s)|_0^{2}ds\Big)^{\frac12} \\
&\leq   2 \Big( MTk_{1}(1 +N_2(T) {\|\gamma\|^2}) \big(\widetilde{N}_{2}(T,M)+ \hat{N}_{2}(T,M)\big)\Big)^{\frac12}   2^{-\frac{n}{2}}.
\end{align*}
Thus, by Gronwall's inequality we arrive at
\begin{equation*}
\widetilde{\mathbb{E}}  \sup_{0\leq t\leq T} \| \widetilde{V}^{\varepsilon}(h_{\varepsilon}, t; h, t)\|^{2} \leq C_2\Big(\varepsilon +  a(\varepsilon)+ \frac{\varepsilon}{a^{2}(\varepsilon)}+ 2^{-\frac{n}{4}} + J_{3}(n,\ve)\Big),
\end{equation*}
 and setting $\varepsilon \to 0$ and then $n\to \infty$ we obtain \eqref{expectationlimit}.\\
\indent We conclude that from any sequence $\{h_{\varepsilon} \}_{ \varepsilon\in(0,\varepsilon_{0})}$, which converges weakly to $h$, one can extract a subsequence $\{h_{\varepsilon_k} \}_{k\in{\mathbb N}}$ such that $\{v^{\varepsilon_k}_{h_{\varepsilon_k}}\}_{k\in{\mathbb N}} $   converges in distribution to the same limit $X^h$ in space $\mathcal{C}([0,T];H)$ as $k\to \infty$. This implies that the whole sequence $\{v^{\varepsilon}_{h_{\varepsilon}}\}_{\varepsilon\in(0,\varepsilon_{0})} $   converges in distribution to $X^h$ in space $\mathcal{C}([0,T];H)$,   as $\varepsilon \to 0$.
\end{proof}
\begin{lemma}\label{LEM3}
Let $X^{h} $ be the unique solution to \eqref{Xhcontrolled} corresponding to $h\in S_{M}$.  Then the set $\{X^{h}: h\in S_{M}\}$ is compact in $\mathcal{C}([0,T];H)$.
\end{lemma}
\begin{proof}
Since $S_{M}$ is a compact set under the weak topology of $L^2({0,T};H_0)$, we show that $h\mapsto X^{h}$ is a continuous map.
Let $h_{\varepsilon},h\in S_{M},\ve>0,$ be such that $\{h_{\varepsilon}\}_{\ve\ge 0}$ converges weakly to $h$ in the space $L^2(0,T;H_0)$. We have for $t\in[0,T]$
\begin{eqnarray*}
\|X^{h_{\varepsilon}}(t)- X^{h}(t)\|^{2} &=& 2 \text{Im} \int_{0}^{t}  {\langle A X^{h_{\varepsilon}}(s)- AX^{h}(s),X^{h_{\varepsilon}}(s)- X^{h}(s)\rangle }ds\\
&& +2  \text{Re} \int_{0}^{t} \left(\mathcal{U}(s)\left(X^{h_{\varepsilon}}(s)-X^{h}(s)\right), X^{h_{\varepsilon}}(s)-X^{h}(s)\right)ds\nonumber\\
&&+ 2  \text{Re}\int_{0}^{t} \left(g(s,u^{0}(s))(h_{\varepsilon}(s)-h(s)), X^{h_{\varepsilon}}(s)-X^{h}(s)\right)ds\nonumber\\
&\leq& 2\sqrt{k_6} \int_{0}^{t}  \hspace{.08cm} \|X^{h_{\varepsilon}}(s)-X^{h}(s)\|^{2}ds\nonumber\\
&&+ 2 \int_{0}^{t} \text{Re}\left(g(s,u^{0}(s))(h_{\varepsilon}(s)-h(s)), X^{h_{\varepsilon}}(s)-X^{h}(s)\right)ds.\nonumber
\end{eqnarray*}
We   apply  a  {similar discretization technique as} in the proof of  Lemma   \ref{LEM2}, analogous to \eqref{expectationlimit}, to obtain
\begin{equation*}
 \lim_{\ve \to 0}\sup_{0\leq t\leq T} \|X^{h_{\varepsilon}}(t)-X^{h}(t )\|^{2} =0.
\end{equation*}
Hence, $\{X^{h}: h\in S_{M}\}$ is compact, being the continuous image of a compact set.
\end{proof}
Thus, we obtain by \cite[Theorem 6]{BDM}, the MDP in Theorem 1 with rate function given by \eqref{rate}.
\section{Central Limit Theorem} \label{sec6}
To achieve the central limit theorem, Theorem \ref{TH3},   we prove that $\{\hat{Z}^{\varepsilon} \}_{ \varepsilon\in(0,\varepsilon_{0})}$ (see \eqref{Zhat}) converges in probability to $Z^0 $ (see \eqref{Vequation}) as $\varepsilon\rightarrow 0$.
Observe that $\hat{Z}^{\varepsilon} $ is the unique variational solution to
\begin{equation}\label{ZCLT}
\hat{Z}^{\varepsilon}(t)= -i \hspace{.08cm} \int_{0}^{t}A\hat{Z}^{\varepsilon}(s) ds + \int_{0}^{t}\mathcal{U}(s)\hat{Z}^{\varepsilon}(s)ds + \int_{0}^{t}  {\hat{g}}(s,\hat{Z}^{\varepsilon}(s))dW(s), t\in[0,T], \mbox{ a.e. } \omega\in \Omega,
\end{equation}
 {Note that $\hat{Z}^{\varepsilon}$ is $v^\ve$, the solution of \eqref{MDPequation} for the special case
 {$a(\varepsilon)=1$}, where for this value of $a(\epsilon)$ we obtain $\hat{g}(s, \hat{Z}^{\varepsilon}(s))$ in \eqref{ZCLT} based on \eqref{MDPg}.}
Hence, by \eqref{vk} we have
\begin{equation} \label{Zp}
\mathbb{E} \sup_{0\leq t\leq T} \|\hat{Z}^{\varepsilon}(t)\|^{2p}\leq {\bar N}_{2p}(T), \mbox{ for }1\le p<\infty,
\end{equation}
and by \eqref{vVnorm}
$$\mathbb{E}\sup_{0\leq t\leq T} \|\hat {Z}^{\varepsilon}(t)\|_{V}^{2}  \leq \bar K (T).$$
Thus, the Chebyshev inequality gives
\begin{equation}\label{Chebyshev}
\lim_{N\rightarrow \infty} \sup_{0<\varepsilon<\varepsilon_{0}} P\left(\sup_{0\leq t\leq T} \frac{\|u^{\varepsilon}(t)-u^{0}(t)\|}{\sqrt{\varepsilon}} > N\right)=0.
\end{equation}
Observe that $Z^0$ in \eqref{Vequation} is $\hat{Z}^{0}$, the solution of \eqref{ZCLT} in the special case $\ve=0$ and $a(\ve)=1$ and as above we have
$$
\mathbb{E} \sup_{0\leq t\leq T} \|Z^0(t)\|^{2p}\leq {\bar N}_{2p}(T) \mbox{ for } {1\leq p<\infty},
\mbox{ and } \mathbb{E}\sup_{0\leq t\leq T} \|{Z}^0(t)\|_{V}^{2}  \leq \bar K (T).
$$
\subsection{Proof of Theorem \ref{TH3}:}
We use \eqref{ZCLT}, \eqref{Vequation},   It\^o's formula and \eqref{opA} to obtain
\begin{align*}
&\mathbb{E} \sup_{0\leq s\leq t} \|\hat{Z}^{\varepsilon}(s)-Z^0(s)\|^{2} \leq { 2 \hspace{.08cm} \mathbb{E} \hspace{.08cm}  \int_{0}^{t}\left| \left(\mathcal{U}(s)(\hat{Z}^{\varepsilon}(s)-Z^0(s)), \hat{Z}^{\varepsilon}(s)-Z^0(s)\right)\right|ds}\\
&+ 2\hspace{.08cm} \mathbb{E} \sup_{0\leq s\leq t}  {\Big|} \int_{0}^{s}  \left(\hat{g}(r,\hat{Z}^{\varepsilon}(r))-g(r,u^{0}(r))dW(r), \hat{Z}^{\varepsilon}(r)-Z^0(r)\right) {\Big|}\\
&+ \mathbb{E}\int_{0}^{t} \|\hat{g}(s, \hat{Z}^{\varepsilon}(s))-g(s,u^{0}(s))\|_{L_{2}}^{2}ds \mbox{ for  }t\in[0,T].
 \end{align*}
 Using  the Burkholder-Davis-Gundy inequality,  \eqref{MDPg},  \eqref{(2.8)} and \eqref{Zp}, we continue for each $t\in[0,T]$
 \begin{align*}
 &\mathbb{E} \sup_{0\leq s\leq t} \|\hat{Z}^{\varepsilon}(s)-Z^0(s)\|^{2}\leq 2\hspace{.08cm}\sqrt{k_6}\hspace{.08cm} \mathbb{E}\int_{0}^{t} \|\hat{Z}^{\varepsilon}(s)-Z^0(s)\|^{2}ds + \frac{1}{2}\hspace{.08cm}\mathbb{E} \sup_{0\leq s\leq t} \|\hat{Z}^{\varepsilon}(s)-Z^0(s)\|^{2} \\
 &+ \varepsilon 18k_{3}\hspace{.08cm} \mathbb{E}\int_{0}^{t} \|\hat{Z}^{\varepsilon}(s)\|^{2}ds\\
&\leq  2\sqrt{k_6} \hspace{.08cm} \mathbb{E} \int_{0}^{t} \sup_{0\leq r \leq s} \|\hat{Z}^{\varepsilon}(r)-Z^0(r)\|^{2}ds  + \frac{1}{2}\sup_{0\leq s\leq t} \|\hat{Z}^{\varepsilon}(s)-Z^0(s)\|^{2}+ 18\hspace{.08cm} \ve k_{3}\hspace{.08cm}T \hspace{.08cm}\bar N_2(T),
\end{align*}
to achieve by Gronwall's lemma
\begin{equation*}
\mathbb{E}\sup_{0\leq t\leq T} \|\hat{Z}^{\varepsilon}(t)-Z^0(t)\|^{2} \leq  C_1(T, k_{3}, k_6, \bar N_2(T)) \varepsilon.
\end{equation*}
We conclude
$$ \lim_{\ve\to0} \mathbb{E}\sup_{0\leq t\leq T} \|\hat{Z}^{\varepsilon}(t)-Z^0(t)\|^{2}=0,$$
which implies to convergence in probability as stated in Theorem \ref{TH3}.


\appendix\section{}

 \subsection{Finite dimensional approximations}  \label{5.1} First we recall a few results from functional  analysis.
 Denote by
$\{\mu_k\}_{k}$ the increasing sequence of real eigenvalues and by $\{\varphi_k\}_{k }$ the corresponding eigenfunctions of the operator $A$ defined in \eqref{eqA}. The eigenfunctions $\{\varphi_k\}_{k}$ form an orthonormal system in $H$ and they are orthogonal in $V$. Notice that for all   $v \in V$  we have
\begin{equation}
  Av = \sum_{k=1}^{\infty} \mu_k ({v},{\varphi_k})   \varphi_k\,,\hspace{.2cm} \langle {Av},{v}\rangle = \sum_{k=1}^{\infty} \mu_k \left| ({v},{\varphi_k}) \right|^2=\|v\|_V^2, \hspace{.2cm}   \mbox{Im}  \langle Av,v\rangle=0. \label{opA}
 \end{equation}
 For each {$n\in \mathbb{N}$}, let $H_n= \text{span}\{\varphi_{1}, \varphi_{2}, ..., \varphi_{n}\}$ be the finite dimensional space equipped with the norm induced from $H$ and let $\pi_{n}: H\rightarrow H_{n}$ be the finite dimensional projection of $H$ onto $H_{n}$ defined by
\begin{equation}\pi_{n}u:= \sum\limits_{k=1}^{n} (u, \varphi_{k}) \varphi_{k}.\label{proj}
\end{equation}
We have for all $u,\hat u\in H$
\begin{equation} \label{new*00}
 \sum_{k= 1}^{n} |({u},{\varphi_k})|^2 =\| \pi_n u\|^2 \le \|u\|^2,\quad
(\hat u,\pi_n u)=(\pi_n \hat u,\pi_n u),
 \end{equation}
\begin{equation} \label{new*1}
 \sum_{k=n+1}^{\infty} |({u},{\varphi_k})|^2 =\|(I-\pi_n)u\|^2 \le \|u\|^2,
 \end{equation}
\begin{equation} \label{new*0}
  (\hat u,(I-\pi_n)u) =
	((I-\pi_n)\hat u,(I-\pi_n)u).
 \end{equation}
 For all $n\in{\mathbb N}, v,\hat v\in V $  we recall that  $A\pi_n v= \sum\limits_{k=1}^{n} \mu_{k}(v,\varphi_{k})\varphi_{k}\in H_n$,
\begin{equation}
\| \pi_n v\|_V\le \|v\|_V \mbox{ and }
(\hat v,A\pi_n v)=(\pi_n \hat v, A\pi_nv) \le \| \pi_n\hat v\|_V \| \pi_n v\|_V\le \|  \hat v\|_V \| \pi_n v\|_V . \label{ineq}
\end{equation}
Notice that for any $v, \hat v\in H$ and $m,n\in \mathbb{N}$ with $m<n$,
\begin{align}\label{equality}
\left( \hat v, A (\pi_{n}-\pi_{m})  v\right) &= \left((\pi_{n}-\pi_{m})\hat v, A (\pi_{n}-\pi_{m})  v\right)\nonumber
\\ &\le \|(\pi_{n}-\pi_{m})\hat v\|_V\|(\pi_{n}-\pi_{m})  v\|_V \le \|\hat v\|_V\|(\pi_{n}-\pi_{m})  v\|_V.
\end{align}

 \subsection{Some properties of convergent series:}  \label{5.2} For a convergent series
 $ \sum\limits_{k\ge 1} a_k$ with   positive terms,  the remainder of the series satisfies
$$
 \lim_{n\to\infty}\sum\limits_{k=n+1}^\infty a_k  =0,
$$
and the sequence of partial sums $ \left\{ \sum\limits_{k= 1}^n   a_k\right\}_n$  is a Cauchy sequence.
\medskip

 \subsection{Proof of estimate (\ref{u0sup}):} \label{5.3}
Similar to the proof of Theorem \ref{TH4}  we  derive (\ref{u0sup}).
Recall that $u^0\in  \mathcal{C}([0,T];H))\cap L^{2}( 0,T;V)$ is the solution of
\begin{equation}\label{u00}
(u^{0}(t),v)= (\gamma,v)  -  i \!\! \int_{0}^{t} \langle Au^{0}(s),v\rangle ds   +   \!  \int_{0}^{t} \!\!(\mathcal{U}(s)u^{0}(s),v)ds
\end{equation}
 for every $t\in[0,T],\ v\in V.$
 For each $k\in \mathbb{N}$
    we  have
\begin{align}\label{new*99}
 |(u^0(t), \varphi_{k})|^{2} =& |(\gamma, \varphi_{k})|^{2}+   2  \text{Im} \int_{0}^{t}(Au^0(s), \varphi_{k})(\overline{u^0(s), \varphi_{k}})ds \nonumber\\ &+ 2 \text{Re} \int_{0}^{t} \left(\mathcal{U}(s)u^0(s), \varphi_{k}\right) (\overline{u^0(s),\varphi_{k}})ds , \mbox{ for every }t\in[0,T]  ,
\end{align} where $\{\varphi_k\}_{k}$ was defined in Subsection \ref{5.1}.
Then, for
$\pi_n u^0$, the finite dimensional projection of $u^0$, see \eqref{proj}, we write
\begin{align} \label{new*49}
  \|(I-\pi_n) u^0(t)\|^{2} = \|(I-\pi_n) \gamma\|^{2} +   2  \text{Re} \int_{0}^{t} \left( \mathcal{U}(s)u^0(s) ,(I- \pi_n) u^0\right) ds,
\end{align}
for every   $n\in \mathbb{N}$,  $t\in[0,T]$. Then by using \eqref{K5}   we obtain
{\begin{align}
   &\sup_{0\leq t\leq T}   \|(I - \pi_{n} ) u^0(t)\|^{2} \le  \|(I-\pi_n) \gamma\|^{2} +
    2  \int_{0}^{T} \big| (  \mathcal{U}(s) u^0(s)  ,
		(I- \pi_{n})   u^0(s))\big| ds \nonumber
\\ & \le  \|(I-\pi_n) \gamma\|^{2} +  k_6\Big(  \sup_{s\in[0,T]} \| u^0(s)\|^{2} \Big)^{\frac12}\Big( \int_{0}^{T}\!\!  \left\| (I-\pi_{n}) u^{0}(s)\right\|^{2}ds\Big)^{\frac12} . \label{u1}
 \end{align}
}We use \eqref{u0}, \eqref{new*1} and Subsection \ref{5.2} to arrive at
 \begin{equation}
\lim_{n\rightarrow \infty}   \sup_{0\leq t\leq T}  \|(I-\pi_{n})u^0(t)\|^{2}  =0. \label{conv19}
\end{equation}
Multiplying both sides of \eqref{new*99} by $\mu_{k}$  {(see Subsection \ref{5.1}) }and summing from $k=1$ to $n$, where $n  \in \mathbb{N}$, we have for all $t\in[0,T]$
\begin{align}
    \|  \pi_{n} u^0(t)\|_V^{2} =   \|  \pi_{n} \gamma\|_V^{2} + 2  \text{Re}   \int_{0}^{t} \left(  \mathcal{U}(s)u^0(s), A   \pi_{n} u^0(s) \right)ds .
  \label{new*139}
	\end{align}
Then by \eqref{K5}, \eqref{ineq} and \eqref{uV0} we get the estimate
\begin{align} \label{import0}
&\sup_{0\leq t\leq T}  \| \pi_{n}  u^0(t)\|^{2}_V \le     C_1  \Big( \|\gamma\|^2_V+ \int_{0}^{T} \|u^0(s)\|_{V}^{2}    ds\Big)    \le C_1  \|\gamma\|^2_V\big(1+  N(T) \big)   .
\end{align}
Similarly, by \eqref{new*99}, \eqref{K5} and \eqref{equality}  we write for each
  $n>m$,  $n,m \in \mathbb{N}$ and every $t\in[0,T]$
\begin{align*}
&\sup_{0\leq t\leq T} \|(\pi_{n}-\pi_m) u^0(t)\|^{2}_V \le \|(\pi_n-  \pi_m) \gamma\|_V^{2}+ 2
 \int_{0}^{T}\big|( \mathcal{U}(s)u^0(s) ,A  (\pi_{n}-\pi_{m})u^0(s)  )\big|ds \\
   &\leq \|(\pi_n-  \pi_m) \gamma\|_V^{2}+ \Big(C_3\int_{0}^{T} \| u^0(s)\|^{2}_Vds\Big)^{\frac12} \Big( \int_{0}^{T}
 \|(\pi_{n}-\pi_{m})u^0(s)\|_{V}^{2}ds \Big)^\frac12 .
\end{align*}
Estimate \eqref{uV0} and the Cauchy criterion from Subsection \ref{5.2}
yield
$$\lim \limits_{\substack{m \to  \infty\\n \to \infty}} \hspace{.08cm} \sup_{0\leq t\leq T} \|(\pi_{n}-\pi_{m})u^0(t)\|_{V}^{2}=0.$$
 {Hence $\{\pi_{n}u^0 \}_{n}$ is a Cauchy sequence in   $ \mathcal{C}([0,T];V) $ and
 there exists  $\hat u^0 \in   \mathcal{C}([0,T];V) $ to which $\{\pi_{n}u^0 \}_{n}$ converges in $ \mathcal{C}([0,T];V) $.
We have the following  inequality
\begin{equation*}
 \sup_{0\leq t\leq T} \|u^0(t)- \hat{u}^{0}(t)\|^{2} \leq  2\sup_{0\leq t\leq T} \|(I-\pi_{n})u^0(t)\|^{2} + 2 k_0\sup_{0\leq t\leq T} \|\hat{u}^{0}(t)-\pi_{n}u^{0}(t)\|^{2}_V,
\end{equation*}
where $k_0$ is the embedding constant of $V\hookrightarrow H$. Then by applying \eqref{conv19} $u^0(t)= \hat{u}^{0}(t)$ for all $t\in [0,T]$.}
By  \eqref{import0} we obtain
\begin{equation*}
\sup_{0\leq t\leq T} \|u^0(t)\|_{V}^{2} =  \sup_{0\leq t\leq T} \|\hat{u}^{0}(t)\|_{V}^{2} = \lim_{n\rightarrow \infty} \sup_{0\leq t\leq T} \| \pi_{n}u^0(t)\|^{2}_V \leq K(T)  \|\gamma\|^2_V ,
\end{equation*}
where $K(T)=C_1 \big(1+    N (T) \big)$.
\medskip

\subsection{Proof of estimate (\ref{vk}):} \label{5.4}
\indent Note that by \eqref{(2.1)}, \eqref{uHbound} for $u^\ve$ and \eqref{u0} for $u^0$ we have that
$\mathbb{E} \displaystyle\sup_{0\leq t\leq T} \|v^{\varepsilon}(t)\|^{2p}<\infty.$  {We use} \eqref{MDPequation}, It\^o's formula, \eqref{opA}, take supremum on time, then the expectation, and use the Burkholder-Davis-Gundy, Young  and Cauchy-Schwarz inequalities, to write for each $t\in[0,T]$  and $1\leq p <\infty$
\begin{align*}
&\mathbb{E}\sup_{0\leq s\leq t}\|v^{\varepsilon}(s)\|^{2p} \leq 2p \hspace{.08cm} \mathbb{E} \int_{0}^{t}\Big| \left(\mathcal{U}(s) v^{\varepsilon}(s), v^{\varepsilon}(s)\right)\Big|\|v^{\varepsilon}(s)\|^{2(p-1)}ds \\
&+ 2p\hspace{.08cm}  a(\varepsilon)\hspace{.08cm} \mathbb{E}\sup_{0\leq s\leq t}\Big|\int_{0}^{s}  \left(\hat{g}(r,v^{\varepsilon}(r))dW(r), v^{\varepsilon}(r)\right) \|v^{\varepsilon}(r)\|^{2(p-1)}\Big|  \\
&+ p  \hspace{.08cm}a^{2}(\varepsilon) \hspace{.08cm} \mathbb{E} \int_{0}^{t} \|\hat{g}(s,v^{\varepsilon}(s))\|_{L_{2}}^{2} \|v^{\varepsilon}(s)\|^{2(p-1)}ds  &\\
&+ 2p(p-1)\hspace{.08cm} a^{2}(\varepsilon)\hspace{.08cm}  \mathbb{E}\int_{0}^{t} \|\hat{g}(s,v^{\varepsilon}(s))\|_{L_{2}}^{2} \|v^{\varepsilon}(s)\|^{2(p-1)}ds&\\
&\leq 2p\hspace{.08cm}\sqrt{k_6} \hspace{.08cm} \mathbb{E} \int_{0}^{t} \|v^{\varepsilon}(s)\|^{2p}ds+ 6 pa(\varepsilon)  \hspace{.08cm}\mathbb{E}\left(\int_{0}^{t} \|\hat{g}(s,v^{\varepsilon}(s))\|_{L_{2}}^{2}  \|v^{\varepsilon}(s)\|^{4p-2} ds\right)^{\frac12}\\
& +  p(2p-1)\hspace{.08cm}a^{2}(\varepsilon) \hspace{.08cm} \mathbb{E} \int_{0}^{t} \|\hat{g}(s,v^{\varepsilon}(s))\|_{L_{2}}^{2} \|v^{\varepsilon}(s)\|^{2(p-1)}ds \\
&\leq \frac{1}{2} \hspace{.08cm} \mathbb{E}\sup_{0\leq s\leq t} \|v^{\varepsilon}(s)\|^{2p} +  C_1(p)\hspace{.08cm} a^{2}(\varepsilon)\hspace{.08cm}  \mathbb{E}\int_{0}^{t} k_{1}\left(1+ \left\|\frac{\sqrt{\varepsilon}}{a(\varepsilon)} v^{\varepsilon}(s) + u^{0}(s)\right\|^{2p}\right)ds  \\
&+ \left(2p\sqrt{k_{6}}+ C_2(p)a^{2}(\varepsilon)\right)\hspace{.08cm} \mathbb{E} \int_{0}^{t}\|v^{\varepsilon}(s)\|^{2p}ds \\
&\leq \frac{1}{2}  \mathbb{E}\sup_{0\leq s\leq t} \|v^{\varepsilon}(s)\|^{2p} +a^{2}(\varepsilon) C_3(p,k_1,N_{2}(T),\|\gamma\|,T) +
  C_4(p,k_1,k_6)\hspace{.08cm} \mathbb{E}\int_{0}^{t} \|v^{\varepsilon}(s)\|^{2p}ds,
\end{align*}
where we used  \eqref{epsi}. Then Gronwall's inequality yields
 {\begin{equation*}
\mathbb{E}\sup_{0\leq t\leq T} \|v^{\varepsilon}(t)\|^{2p} \leq a^{2}(\varepsilon) \bar N_{2p}(T),
\end{equation*}
where
\begin{equation}
\label{veps}
\bar N_{2p}(T)=2 C_3(p,k_{1},N_{2p}(T),\|\gamma\|,T) e^{2 C_4(p,k_1,k_6) T}  .
\end{equation}}
\medskip

\subsection{Proof of estimate  (\ref{vh2pinequality}):} \label{5.5}
For $n\ge 1$ we consider the stopping time  $\tau_n:= \inf\{t: \|v_{h}^{\varepsilon}(t)\|^{2}>n\}$, and for $h\in \mathcal{P}_{M}$ we introduce the process
\begin{equation} \label{theta2}
\theta(t)=\exp\Big(-\displaystyle c\int_0^t\big(1+ |h(s) |^{2}_0\big)ds\Big) , \ t\in[0,T],
\end{equation}
where $c>0$ is a constant whose value will be indicated later.
We use   \eqref{MDPstochastic}, the It\^o formula, \eqref{opA},       and apply the Burkholder-Davis-Gundy, Young's and Cauchy-Schwarz inequalities, along with  \eqref{(2.6)}, \eqref{K5}, \eqref{epsi} and \eqref{u0}   to write for each $t\in[0,T]$, $1\leq p<\infty$,
\begin{align*}
&\mathbb{E} \sup_{0\leq s\leq t\wedge \tau_n}  \theta(s)\|v_{h}^{\varepsilon}(s)\|^{2p}
\leq   2p  \hspace{.08cm} \mathbb{E}  \int_{0}^{t\wedge \tau_n}\theta(s) \left| \left(\mathcal{U}(s) v_{h}^{\varepsilon}(s), v_{h}^{\varepsilon}(s)\right)\right| \|v_{h}^{\varepsilon}(s)\|^{2(p-1)}ds\\
&  + 2p\hspace{.08cm} a(\varepsilon)\hspace{.08cm}   \mathbb{E} \sup_{0\leq s\leq t\wedge \tau_n} \Big|\int_{0}^{s}  \theta(r) \left(\hat{g}(r, v_{h}^{\varepsilon}(r)) {dW(r)}, v_{h}^{\varepsilon}(r)\right) \|v_{h}^{\varepsilon}(r)\|^{2(p-1)}\Big| \\
&  + 2p\hspace{.08cm}  \mathbb{E}  \int_{0}^{t\wedge \tau_n}  \theta(s)  \Big|\left(\hat{g}(s, v_{h}^{\varepsilon}(s)) h(s), v_{h}^{\varepsilon}(s)\right)\Big| \|v_{h}^{\varepsilon}(s)\|^{2(p-1)}ds
\\& +    p(2p-1)\hspace{.08cm}a^2 (\varepsilon)\hspace{.08cm}    \mathbb{E} \int_{0}^{t\wedge \tau_n} \theta(s) \|\hat{g}(s, v_{h}^{\varepsilon}(s))\|_{L_{2}}^{2} \|v_{h}^{\varepsilon}(s)\|^{2(p-1)}ds \\& -c \mathbb{E}  \int_{0}^{t\wedge \tau_n}  \theta(s) \|v_{h}^{\varepsilon}(s)\|^{2p}(1+ |h(s)|^2_0) ds \\
&\leq  2p\hspace{.08cm}\sqrt{k_6} \hspace{.08cm}\mathbb{E} \int_{0}^{t} \theta(s)\|v^{\varepsilon}(s)\|^{2p}ds
  +6 pa(\varepsilon)\hspace{.08cm}  \mathbb{E}\left(\int_{0}^{t\wedge \tau_n}\!\!\! \theta^2 (s) \|\hat{g}(s,v^{\varepsilon}(s))\|_{L_{2}}^{2}  \|v^{\varepsilon}(s)\|^{4p-2} ds\right)^{\frac12}
\\ &+2p\hspace{.08cm}  \mathbb{E}  \int_{0}^{t\wedge \tau_n}  \theta(s)  \|\hat{g}(s, v_{h}^{\varepsilon}(s))\|_{L_2} |h(s)|_0 \|v_{h}^{\varepsilon}(s)\|^{2 p-1 }ds\\
& +    a^2 (\varepsilon)\hspace{.08cm} C_1(p,k_1,N_2(T),\|\gamma\|)   \mathbb{E} \int_{0}^{t\wedge \tau_n} \theta(s) (1+ \|v_{h}^{\varepsilon}(s)\|^{2p})ds\\
&-c\hspace{.08cm} \mathbb{E}  \int_{0}^{t\wedge \tau_n}  \theta(s) \|v_{h}^{\varepsilon}(s)\|^{2p}(1+ |h(s)|^2_0) ds\\
&\le
\frac{1}{2}\hspace{.08cm} \mathbb{E} \sup_{0\leq s\leq t\wedge \tau_n}\theta(s)  \|v_{h}^{\varepsilon}(s)\|^{2p} + C_2(p,k_1, N_{2}(T),\|\gamma\|,T)  \\& + C_3(p,k_1,k_6,N_2(T),\|\gamma\|)\hspace{.08cm}  \mathbb{E}\int_{0}^{t\wedge \tau_n}\theta(s)  \|v_{h}^{\varepsilon}(s)\|^{2p}(1+ |h(s)|^2_0) ds\\
& -c \hspace{.08cm}\mathbb{E}  \int_{0}^{t\wedge \tau_n}  \theta(s) \|v_{h}^{\varepsilon}(s)\|^{2p}(1+ |h(s)|^2_0) ds.
\end{align*}
  We consider in \eqref{theta2} the constant
$c:=C_3(p,k_1,k_6, N_{2}(T),\|\gamma\|).$
Then we use that $h\in \mathcal{P}_{M}$ and we take $n\to\infty$ to conclude
\begin{align*}
\mathbb{E} \sup_{0\leq s\leq T }  \theta(s)\|v_{h}^{\varepsilon}(s)\|^{2p}
\leq 2C_2(p,k_1, N_{2p}(T)).
\end{align*}
By \eqref{theta2} we obtain
$$
\mathbb{E}\sup_{0\leq t\leq T} \|v_{h}^{\varepsilon}(t)\|^{2p} \leq \widetilde{N}_{2p}(T,M),$$
where
\begin{equation}
\widetilde{N}_{2p}(T,M)=2C_2(p,k_1, N_{2}(T),\|\gamma\|,T)\exp\left( \displaystyle c(T+M)\right)  .\label{A.11}
\end{equation}

\subsection{Proof of estimate (\ref{Xhinequality}):} \label{5.6}
 We have by \eqref{Xhcontrolled}, \eqref{opA}, \eqref{K5}, \eqref{(2.6)} and \eqref{u0}
\begin{align*}
&\sup_{0\leq s\leq t} \|X^{h}(s)\|^{2} \leq  2     \int_{0}^{t} \left| (\mathcal{U}(s)X^{h}(s), X^{h}(s) )\right|  ds +2    \int_{0}^{t} \left| (g(s,u^{0}(s))h(s), X^{h}(s) ) \right|  ds \\
&\leq 2 \sqrt{k_6} \int_{0}^{t} \|X^{h}(s)\|^{2}ds +
 2\int_{0}^{t} \|g(s,u^{0}(s))\|_{L_{2}} \hspace{.08cm} |h(s)|_0 \hspace{.08cm}  \|X^{h}(s)\|ds \\
&\le  \big(2 \sqrt{k_6}  +  k_1 + k_1 N_2(T)\|\gamma\|^2 \big)   \int_{0}^{t} \|X^{h}(s)\|^{2}ds +   \int_{0}^{t}    |h(s)|_0^2ds, \mbox{ for each } t\in[0,T]\,.
\end{align*}
Then Gronwall's lemma and $h\in {\mathcal P}_{M} $ lead to
\begin{equation*}
\sup_{0\leq s\leq T} \|X^{h}(s)\|^{2} \leq Me^{T(2 \sqrt{k_6}  +  k_1 +
k_1 N_2(T)\|\gamma\|^2) }\,,
\end{equation*}
and taking both sides of the above inequality to the $p$th power, with $1\leq p<\infty$, we obtain \eqref{Xhinequality}, where
\begin{equation}\label{A.12}
\hat{N}_{2p}(T,M)= M^pe^{pT(2 \sqrt{k_6}  +  k_1 +
k_1 N_2(T)\|\gamma\|^2) }.
\end{equation}
\section*{Acknowledgements}
 The authors are grateful to the anonymous referee for his/her time in reading the article in depth and providing detailed comments and suggestions that helped improve the paper.  P. Fatheddin is also thankful to Prof. Annie Millet for helpful conversations on the time discretization technique applied in the weak convergence approach.

\end{document}